\documentclass[11pt,twoside,a4paper]{amsart}
\usepackage{amsfonts, amsthm, amssymb, amsmath, stmaryrd}
\usepackage{mathrsfs,array}
\usepackage{eucal,times,color,enumerate,accents}
\usepackage[all]{xy}
\usepackage{graphicx}
\usepackage{url}
\usepackage{enumitem}
\usepackage{amssymb}
\usepackage{comment} %to comment out whole sections
\usepackage{mathtools}% mathrlap or phantom?
\usepackage[utf8]{inputenc}
\usepackage[T1]{fontenc} %french letters directly in code

\usepackage{enumitem}% http://ctan.org/pkg/enumitem   allows nested enumerations
\usepackage[normalem]{ulem} % either use this (simple) or

\usepackage[colorinlistoftodos]{todonotes}

\usepackage{graphicx}
\usepackage{tikz-cd}
\usetikzlibrary{calc}
\usetikzlibrary{matrix,arrows,decorations.pathmorphing}

\usepackage{stmaryrd}
\usepackage{trimclip}

\usepackage{ulem} %underlining is wrapping

\usepackage{indentfirst} %wciecia (ang. "ident") takze na poczatkach rozdzialow

\makeatletter
\DeclareRobustCommand{\shortto}{%
  \mathrel{\mathpalette\short@to\relax}%
}

\newcommand{\short@to}[2]{%
  \mkern2mu
  \clipbox{{.5\width} 0 0 0}{$\m@th#1\vphantom{+}{\shortrightarrow}$}%
  }
\makeatother
\DeclareRobustCommand{\SkipTocEntry}[4]{} %for removing chosen subsections from ToC. The previous definition will not work properly if the hyperref package is being used. The hyperref package redefines TOC entries to have 5 arguments, so use this definition instead:  \DeclareRobustCommand{\SkipTocEntry}[5]{} 

\input xy
\xyoption{all}

\newcommand{\calO}{\mathcal{O}}

\newcommand{\calC}{\mathcal{C}}

\newcommand{\bbZ}{\mathbb{Z}}

\newcommand{\bbC}{\mathbb{C}}
\newcommand{\bbQ}{\mathbb{Q}}
\newcommand{\bbP}{\mathbb{P}}

\newcommand{\piet}{\pi_1^{\et}}
\newcommand{\pipet}{\pi_1^{\mathrm{pro\et}}}
\newcommand{\pisga}{\pi_1^{\mathrm{SGA3}}}

\newcommand{\Spec}{\mathrm{Spec}}
\newcommand{\Gal}{\mathrm{Gal}}

\newcommand{\Aut}{\mathrm{Aut}}

\newcommand{\GL}{\mathrm{GL}}
\newcommand{\id}{\mathrm{id}}
\newcommand{\et}{\mathrm{\acute{e}t}}
\newcommand{\Et}{\mathrm{\acute{E}t}}

\newcommand{\Image}{\mathrm{Im}}
\newcommand{\Ker}{\mathrm{Ker}}

\newcommand{\rmcolim}{\mathrm{colim}}

\newcommand{\bk}{\bar{k}}
\newcommand{\bs}{\bar{s}}
\newcommand{\bt}{\bar{t}}
\newcommand{\bx}{\bar{x}}

\newcommand{\bxi}{\bar{\xi}}

\newcommand{\shfF}{\mathscr{F}}

\newcommand{\rmSets}{\mathrm{Sets}}
\newcommand{\rmNoohi}{\mathrm{Noohi}}

\newcommand{\rmLoc}{\mathrm{Loc}}

\newcommand{\rmtop}{\mathrm{top}}
\newcommand{\rmim}{\mathrm{im}}
\newcommand{\rmsep}{\mathrm{sep}}

\newcommand{\rmCov}{\mathrm{Cov}}

\newcommand{\rmAut}{\mathrm{Aut}}
\newcommand{\rmFrac}{\mathrm{Frac}}
\newcommand{\rmShv}{\mathrm{Shv}}
\newcommand{\tX}{\widetilde{X}}

\newcommand{\bbH}{\overline{\overline{H}}}
\newcommand{\hbbZ}{\widehat{\mathbb{Z}}}

\newcommand{\tS}{\widetilde{S}}
\newcommand{\tT}{\widetilde{T}}
\newcommand{\tY}{\widetilde{Y}}
\newcommand{\tZ}{\widetilde{Z}}
\newcommand{\tF}{\widetilde{F}}

\newcommand{\tU}{\widetilde{U}}

\newcommand{\ts}{\tilde{s}}

\newcommand{\proet}{\mathrm{pro\et}}

\newcommand{\rarr}{\rightarrow}
\newcommand{\epirarr}{\twoheadrightarrow}
\newcommand{\monorarr}{\hookrightarrow}

\newcommand{\invlim}{\varprojlim}

\begin{document}
\bibliographystyle{alpha}
\newtheorem{theorem}{Theorem}[section]
\newtheorem{proposition}[theorem]{Proposition}
\newtheorem{lemma}[theorem]{Lemma}
\newtheorem{corollary}[theorem]{Corollary}
\newtheorem{claim}[theorem]{Claim}
\newtheorem{definition}[theorem]{Definition}
\newtheorem{conjecture}[theorem]{Conjecture}
\newtheorem{defn}[theorem]{Definition}
\newtheorem{prop}[theorem]{Proposition}
\newtheorem*{theorem*}{Theorem}
\newtheorem*{lemma*}{Lemma}
\newtheorem*{proposition*}{Proposition}

\theoremstyle{definition}
\newtheorem{question}[theorem]{Question}
\newtheorem{answer}[theorem]{Answer}
\newtheorem{remark}[theorem]{Remark}
\newtheorem{example}[theorem]{Example}
\newtheorem{warning}[theorem]{Warning}
\newtheorem{notation}[theorem]{Notation}
\newtheorem{construction}[theorem]{Construction}
\newtheorem{fact}[theorem]{Fact}
\newtheorem{obs}[theorem]{Observation}
\newtheorem{rmk}[theorem]{Remark}

\numberwithin{theorem}{section} %we get Proposition 2.10 instead of 2.x.x and so on

\newcommand{\adjunction}[4]{\xymatrix@1{#1{\ } \ar@<-0.3ex>[r]_{ {\scriptstyle #2}} & {\ } #3 \ar@<-0.3ex>[l]_{ {\scriptstyle #4}}}}

\title{Homotopy Exact Sequence for the Pro-\'Etale Fundamental Group II}
\date{\today}
\author{Marcin Lara}

\address{Instytut Matematyczny PAN, Śniadeckich 8, Warsaw, Poland}
  \email{marcin.lara@impan.pl}


\begin{abstract} The pro-\'etale fundamental group of a scheme, introduced by Bhatt and Scholze, generalizes the usual \'etale fundamental group $\piet$ defined in SGA1 and leads to an interesting class of "geometric coverings" of schemes, generalizing finite \'etale covers.
We prove exactness of the general homotopy sequence for the pro-\'etale fundamental group, i.e. that for a geometric point $\bs$ on $S$ and a flat proper morphism $X \rarr S$ of finite presentation whose geometric fibres are connected and reduced, the sequence
\begin{displaymath}
  \pipet(X_{\bs}) \rightarrow \pipet(X) \rightarrow \pipet(S) \rightarrow 1
\end{displaymath}
is "nearly exact".
This generalizes a theorem of Grothendieck from finite \'etale covers to geometric coverings. We achieve the proof by constructing an infinite (i.e. non-quasi-compact) analogue of the Stein factorization in this setting.
\end{abstract}

\maketitle

 \tableofcontents

%%%%%%%%%%%%%%%%%%%%%%%%%%%%%%%%%%%%%%%%%%%%%%%%%%%%%%%%%%%%%%%%%%%%%%%%%%%%%%%%%%%%%%%%%%%%%%%%
\section{Introduction}
In \cite{BhattScholze}, the authors introduced the pro-\'etale topology for schemes. 
Along with the new topology, they defined a new fundamental group -- the pro-\'etale fundamental group. It is defined for a connected locally topologically noetherian scheme $X$ with a geometric point $\bx$ and is denoted $\pipet(X,\bx)$.

In Grothendieck's approach, one takes the category $\mathrm{F\Et}_X$ of finite \'etale covers together with the fibre functor $F_{\bx}$ and obtains an equivalence $\piet(X,\bx) - \mathrm{FSets} \simeq \mathrm{F\Et}_X$,  where $G - \mathrm{FSets}$ denotes \underline{finite} sets with a continuous $G$-action. The pro-\'etale fundamental group generalizes this; there is an equivalence between the category  $\pipet(X,\bx)-\rmSets$ of (possibly infinite) discrete sets with a continuous $\pipet(X,\bx)$-action and a larger class of coverings, namely "geometric coverings", which are defined to be schemes $Y$ over $X$ such that $Y \rarr X$:
\begin{enumerate}
    \item is \'etale (not necessarily quasi-compact!)
    \item satisfies the valuative criterion of properness.
\end{enumerate}
We denote the category of geometric coverings by $\rmCov_X$. An example of a non-finite covering in $\rmCov_X$ can be obtained by viewing an infinite chain of (suitably glued) $\bbP^1_k$'s as a covering of the nodal curve $X=\bbP^1/\{0,1\}$ obtained by gluing $0$ and $1$ on $\bbP^1_k$. The group $\pipet$ generalizes $\piet$ and the more general group $\pisga$ defined in \cite[Chapter X.6]{SGA3vol2}. The name "pro-\'etale" is justified by the fact that there is an equivalence $\rmCov_X \simeq \rmLoc_{X_{\proet}}$, where $\rmLoc_{X_{\proet}}$ denotes the category of locally constant sheaves of discrete sets in $X_{\proet}$.

\addtocontents{toc}{\SkipTocEntry}
\subsection*{The results}
In SGA1, Grothendieck proved the homotopy exact sequence for the \'etale fundamental group.
\begin{theorem*}(\cite[Exp. X, Cor. 1.4 + Cor. 1.8]{SGA1})
  Let $f:X \rarr S$ be a flat proper morphism of finite presentation whose geometric fibres are connected and reduced. Assume $S$ is connected and let $\bs$ be a geometric point of $S$. Let $\bx$ be a geometric point on $X_{\bs}$. Then the sequence 
  \begin{displaymath}
  \piet(X_{\bs},\bx) \rarr \piet(X,\bx) \rarr \piet(S,\bs) \rarr 1
  \end{displaymath}
  is exact.
\end{theorem*}
Our main result generalizes this theorem from $\piet$ to $\pipet$.

\begin{theorem*}(See Thm. \ref{homotopy-exact-general-base})
Let $f : X \rightarrow S$ be a flat proper morphism of finite presentation whose geometric fibres are connected and reduced. Assume that $S$ is Nagata and connected. Let $\bs$ be a geometric point of $S$ and let $\bx$ be a geometric point on $X_{\bs}$.  Then the sequence induced on the pro-\'etale fundamental groups
\begin{displaymath}
\pipet(X_{\bs},\bx) \rightarrow \pipet(X,\bx) \rightarrow \pipet(S,\bs) \rightarrow 1
\end{displaymath}
is nearly exact (see Defn. \ref{definition-weakly-exact}).
\end{theorem*}
The near-exactness means that one needs to take certain kinds of closures (of the images to make them equal the kernels). For exactness in the middle, this cannot be avoided; taking for example $X = \{ZY^2 = W^3 + ZW^2 + Z^3t\} \subset \bbP_R^2$, where $R = \bbC[[t]]$, one sees that the special fiber is the nodal curve described earlier and $\pipet(X_{\bs}) = \bbZ$. On the other hand, the scheme $X$ is normal, which implies that $\pipet(X) = \piet(X) = \hbbZ$. Thus, the sequence reads $\bbZ \rarr \hbbZ \rarr 1 \rarr 1$ (see Remark \ref{remark-that-only-nearly-exact}).

The need for the notion of near-exactness stems from the fact that the group $\pipet$ has a more complicated topology than $\piet$; the groups $\pipet$ belong to the class of Noohi groups. For example, they are not necessarily compact in general. However, near-exactness of the homotopy sequence translates to a statement about the geometric coverings (of $X_{\bs}$, $X$ and $S$), which is a generalization of the theorem of Grothendieck, where one replaces $\mathrm{F\Et}$ by $\rmCov$. In other words, despite the sequence being only nearly exact, in terms of coverings we prove "everything that was to be proven".

Besides $\pipet$ having a more complicated topology, the main difficulties in trying to directly generalize the proof of Grothendieck are as follows:
\begin{itemize}
  \item geometric coverings of schemes (i.e. elements of $\rmCov_X$ defined above) are often not quasi-compact, unlike elements of $\mathrm{F\Et}_X$. Some useful constructions that worked for finite \'etale covers (like the Stein factorization) will fail, unless considerably generalized.
  \item for a connected geometric covering $Y \in \rmCov_X$, there is in general no Galois geometric covering dominating it. Equivalently, there might exist an open subgroup $U < \pipet(X)$ that does not contain an open normal subgroup. This prevents some proofs that would work for $\pisga$ from carrying over to $\pipet$.
\end{itemize}
Indeed, the (near-)exactness in the middle of the homotopy exact sequence boils down to the following statement. For a connected $Y \in \rmCov_X$ such that the structure morphism $Y\times_X X_{\bs} \rarr X_{\bs}$ has a section, there exists $T \in \rmCov_S$ such that $Y \simeq T \times_S X$. For $Y$ finite \'etale, one takes $T$ to be the Stein factorization of $h:Y \rarr S$, i.e. $Y = \underline{\Spec}_S(h_*\calO_{Y})$ and checks that it is (finite) \'etale and has the desired properties. For non-quasi-compact $Y$, this definition usually does not give the correct answer. When trying to directly write $Y$ as a union (or a gluing) of quasi-compact subschemes (e.g. open or closed, or unions of irreducible components) and applying the Stein factorization to each, one quickly runs into problems. We construct the desired $T$ in a different way and call it an "infinite Stein factorization". Our method is as follows; by the equivalence $\rmCov_X \simeq \rmLoc_{X_{\proet}}$, such a $T$ should split completely after a base-change to some large pro-\'etale cover  $\tS$ of $S$. Thus, after base-changing to $\tS$, there is an obvious candidate for $\tT = T \times_S \tS$; namely, $\tT = \sqcup_{i \in I} \tS_i$, i.e. a disjoint union of copies of $\tS$ parametrized by some indexing set $I$. More precisely, $I = \pi_0(Y_{\bs})$. We show that for $\tT$ defined in such a way, there is a map $\tY \rarr \tT$ with the desired properties, and moreover, we have a descent datum with respect to the pro-\'etale cover $\tS \rarr S$. We do this by first showing the theorem for $S$ equal to the spectrum of a strictly henselian ring and then generalizing it to $\tS$ (whose connected components are the spectra of strictly henselian local rings, but $\pi_0(\tS)$ might have a complicated profinite topology).

The precise statement of the existence of the infinite Stein factorization is as follows.
\begin{theorem*}[Thm. \ref{steinforgeomcov}] Let $S$ be a Nagata scheme. Let $X \rarr S$ be as in the theorem above and let $Y \in \rmCov_X$ be connected. Then there exists a connected $T \in \rmCov_S$ and a morphism $g:Y \rarr T$ over $X \rarr S$ such that $g$ has geometrically connected fibres. 

  Moreover, for any two $T_1$, $T_2$ and maps $g_i: Y \rarr T_i$, $i=1,2$, as in the statement, there exists a unique isomorphism $\phi:T_1 \simeq T_2$ in $\rmCov_S$ such that $g_2 = \phi \circ g_1$. 
  \end{theorem*}
  
This article continues the study started in the first part (\cite{PartI}), where we dealt with the K\"unneth formula, general van Kampen theorem and the comparison between "geometric" and "arithmetic" fundamental groups in the case of $\pipet$.

\addtocontents{toc}{\SkipTocEntry}
\subsection*{Acknowledgements}
The results contained in this article are a part of my PhD thesis. I express my gratitude to my advisor H\'el\`ene Esnault for introducing me to the topic and her constant encouragement. I would like to thank my co-advisor Vasudevan Srinivas for his support and suggestions. I am thankful to Peter Scholze for explaining some parts of his work to me via e-mail. I thank Jo\~ao Pedro dos Santos for for his comments and feedback. I owe special thanks to  Fabio Tonini, Lei Zhang and Marco D'Addezio from our group in Berlin for many inspiring mathematical  discussions. I thank Piotr Achinger for his support. During my PhD, I was funded by the Einstein Foundation. 

This  work  is a  part  of  the  project KAPIBARA supported by the funding from the European Research Council (ERC) under the European Union’s Horizon 2020 research and innovation programme (grant agreement No 802787).

\addtocontents{toc}{\SkipTocEntry}
\subsection{Conventions and notations}
\begin{itemize}
    \item For a field $k$, we will use $\bk$ to denote its (fixed) algebraic closure and  $k^{\rmsep}$ or $k^s$ to denote its separable closure (in $\bk$).
    \item The topological groups are assumed to be Hausdorff unless specified otherwise or appearing in a context where it is not automatically satisfied (e.g. as a quotient by a subgroup that is not necessarily closed). We will usually comment whenever a non-Hausdorff group appears.
    \item We assume (almost) every base scheme to be locally topologically noetherian. This does not cause problems when considering geometric coverings, as a geometric covering of a locally topologically noetherian scheme is locally topologically noetherian again - this is \cite[Lm. 6.6.10]{BhattScholze}. Without this assumption, the category of geometric coverings does not behave in the desired way: see \cite[Example 7.3.12]{BhattScholze}. On the other hand, in some proofs we base-change to large pro-\'etale covers (e.g. so-called w-contractible covers) which are usually non-noetherian and some care is needed.
    \item A "$G$-set" for a topological group $G$ will mean a discrete set with a continuous action of $G$ unless specified otherwise. We will denote the category of $G$-sets by $G-\rmSets$.
    \item We will denote the category of (discrete) sets by $\rmSets$.
    \item We will often omit the base points from the statements and the discussion; by \cite[Cor. 3.21]{PartI}, this usually does not change much.
\end{itemize}

\section{Overview of the results in \cite{BhattScholze} and Part I}
We will use the language and results of \cite{BhattScholze}, especially of Chapter 7, as this is where the pro-\'etale fundamental group was defined. We are going to give a quick overview of some of these results below, but we recommend keeping a copy of \cite{BhattScholze} at hand.

\begin{defn}(\cite[Defn. 7.1.1]{BhattScholze})
A Hausdorff topological group $G$ is a \emph{Noohi group} if the natural map induces an isomorphism $G \rarr \rmAut(F_G)$ of topological groups. Here, $F_G : G-\rmSets \rarr \rmSets$ is the forgetful functor. For a discrete set $S$, we endow $\rmAut(S)$ with the compact-open topology and $\rmAut(F_G)$ is topologized using $\rmAut(S)$ for $S \in \rmSets$.
\end{defn}

By \cite[Prop. 7.1.5]{BhattScholze}, a topological group is Noohi if and only if it satisfies the following conditions:
\begin{itemize}
    \item its open subgroups form a basis of open neighbourhoods of $1 \in G$,
    \item it is Ra{\u \i}kov complete.
\end{itemize}
A topological group $G$ is Ra{\u \i}kov complete if it is complete for its two-sided uniformity (see \cite{Dikranjan} or \cite[Chapter 3.6]{AT} for an introduction to the Ra{\u \i}kov completion).
\begin{example}
The following classes of topological groups are Noohi: discrete groups, profinite groups, $\rmAut(S)$ with the compact-open topology for $S$ a discrete set (see \cite[Lm. 7.1.4]{BhattScholze}), groups containing an open subgroup which is Noohi (see \cite[Lm. 7.1.8]{BhattScholze}).

The following groups are Noohi: $\bbQ_\ell$, $\overline{\bbQ_\ell}$ for the colimit topology induced by expressing $\overline{\bbQ_\ell}$ as a union of finite extensions (in
contrast with the situation for the $\ell$-adic topology),  $\GL_n(\bbQ_\ell)$ for the colimit topology (see \cite[Example 7.1.7]{BhattScholze}).
\end{example}

The notion of a Noohi group is tightly connected to a notion of an infinite Galois category, defined in \cite[Defn. 7.2.1]{BhattScholze}. An \emph{infinite Galois category} is a pair $(\calC,F : \calC \rarr \rmSets)$ satisfying certain categorical properties, where $\calC$ is a category. With an additional assumption of \emph{tameness}, it generalizes the notion of a Galois category, introduced by Grothendieck to define $\piet$. The basic example is as follows: if $G$ is a topological group, then $(G-\rmSets,F_G)$ is a tame infinite Galois category. It turns out, that any tame infinite Galois category is of this form. For a pair $(\calC,F)$, one defines $\pi_1(\calC,F) = \Aut(F)$. It is topologized using $\Aut(S)$ with compact-open topology, for $S \in \rmSets$. By \cite[Thm. 7.2.5]{BhattScholze}, if $(\calC,F)$ is an infinite Galois category, then $\pi_1(\calC,F)$ is Noohi and if $(\calC,F)$ is moreover tame, then $F$ induces an equivalence
\begin{displaymath}
  \calC \simeq \pi_1(\calC,F)-\rmSets.
\end{displaymath}

The above formalism is a corrected version of the theory in \cite{Noohi}. It was also studied in \cite[Chapter 4]{Lepage} under different names.

\addtocontents{toc}{\SkipTocEntry}
\subsection*{Pro-\'etale topology and the definition of $\pipet(X)$}
Fix a locally topologically noetherian scheme $X$. 
\begin{defn}
Let $Y \rarr X$ be a morphism of schemes such that:
\begin{enumerate}
    \item it is \'etale (not necessarily quasi-compact!)
    \item it satisfies the valuative criterion of properness.
\end{enumerate}
We will call $Y$ a \emph{geometric covering} of $X$. We will denote the category of geometric coverings by $\rmCov_X$.
\end{defn}
As $Y$ is not assumed to be of finite type over $X$, the valuative criterion does not imply that $Y \rarr X$ is proper (otherwise we would simply get a finite \'etale morphism). For example, for an algebraically closed field $\bk$, the category $\rmCov_{\Spec(\bk)}$ consists of (possibly infinite) disjoint unions of $\Spec(\bk)$ and we have $\rmCov_{\Spec(\bk)} \simeq \rmSets$. More generally, one has:
\begin{lemma}(\cite[Lm. 7.3.8]{BhattScholze}) If $X$ is a henselian local scheme, then any $Y \in \rmCov_X$ is a disjoint union of finite \'etale $X$-schemes.
\end{lemma}

A basic example of a non-finite connected covering in $\rmCov_X$ can be obtained by viewing an infinite chain of (suitably glued) $\bbP^1_k$'s as a covering of the nodal curve $X=\bbP^1/\{0,1\}$ obtained by gluing $0$ and $1$ on $\bbP^1_k$.

Let us choose a geometric point $\bx: \Spec(\bk) \rarr X$ on $X$. This gives a fibre functor $F_{\bx}: \rmCov_X \rarr \rmSets$.
By \cite[Lemma 7.4.1]{BhattScholze}), the pair $(\rmCov_X,F_{\bx})$ is a tame infinite Galois category. Then one defines
\begin{defn}
The \emph{pro-\'etale fundamental group} is defined as
\begin{displaymath}
\pipet(X,\bx) = \pi_1(\rmCov_X,F_{\bx}).
\end{displaymath}
In other words, $\pipet(X,\bx)=\rmAut(F_x)$ and this group is topologized using the compact-open topology on $\rmAut(S)$ for any $S \in \rmSets$.
\end{defn}

One can compare the groups $\pipet(X,\bx)$, $\piet(X,\bx)$ and  $\pi_1^{\mathrm{SGA3}}(X,\bx)$, where the last group is the group introduced in Chapter X.6 of \cite{SGA3vol2}.
\begin{lemma}For a scheme $X$, the following relations between the fundamental groups hold
\begin{enumerate}
    \item The group $\piet(X,\bx)$ is the profinite completion of $\pipet(X,\bx)$.
    \item The group $\pi_1^{\mathrm{SGA3}}(X,\bx)$ is the prodiscrete completion of $\pipet(X,\bx)$.
\end{enumerate}
\end{lemma}
\begin{proof}
This follows from \cite[Lemma 7.4.3]{BhattScholze} and \cite[Lemma 7.4.6]{BhattScholze}.
\end{proof}
As shown in \cite[Example 7.4.9]{BhattScholze}, $\pipet(X,\bx)$ is indeed more general than $\pi_1^{\mathrm{SGA3}}(X,\bx)$. This can be also seen by combining \cite[Ex. 4.5]{PartI} with \cite[Prop. 4.8]{PartI}.

The following lemma is very useful. Recall that, for example, a normal scheme is geometrically unibranch.
\begin{lemma}(\cite[Lm. 7.4.10]{BhattScholze})\label{proetale-of-normal}
If $X$ is geometrically unibranch, then $\pipet(X,\bx) \simeq \piet(X,\bx)$.
\end{lemma}

There is another way of looking at the pro-\'etale fundamental group, which justifies the name "pro-\'etale". 
\begin{defn}
\begin{enumerate}
  \item A map $f : Y \rarr X$ of schemes is called \emph{weakly \'etale} if $f$ is flat and the diagonal $\Delta_f : Y \rarr Y\times_XY$ is flat.
  \item The pro-\'etale site $X_\proet$ is the site of weakly \'etale $X$-schemes, with covers given by fpqc covers.  
\end{enumerate}  
\end{defn}
This definition of pro-\'etale site is justified by a foundational theorem -- part \ref{olivier-item} of the following fact.
\begin{fact}
  Let $f : A \rarr B$ be a map of rings.
  \begin{enumerate}[label=\alph*)]
      \item $f$ is étale if and only if $f$ is weakly étale and finitely presented.
      \item If $f$ is ind-étale, i.e. $B$ is a filtered colimit of étale A-algebras, then $f$ is weakly étale.
      \item \label{olivier-item} (\cite[Theorem 2.3.4]{BhattScholze}) If $f$ is weakly étale, then there exists a faithfully flat ind-étale $g : B \rarr C$ such that $g\circ f$ is ind-étale.
  \end{enumerate}
\end{fact}

\begin{defn}(\cite[Defn. 7.3.1.]{BhattScholze}) 
We say that $F \in \rmShv(X_{\proet})$ is \emph{locally constant} if there exists a cover $\{Y_i \rarr X\}$ in $X_{\proet}$ with $F|_{Y_i}$ constant. We write $\rmLoc_X$ for the corresponding full subcategory of $\rmShv(X_{\proet})$.
\end{defn}

We are ready to state the following important result. As explained in \cite[Ex. 7.3.12]{BhattScholze}, this theorem needs the  assumption of being locally topologically noetherian to hold.
\begin{theorem}(\cite[Lemma 7.3.9.]{BhattScholze}) One has $\rmLoc_X = \rmCov_X$ as subcategories of $\rmShv(X_{\proet})$.
\end{theorem}

\begin{rmk}\label{classical-remark}
There is also a notion of \emph{locally weakly constant} sheaves (\cite[Defn. 7.3.1.]{BhattScholze}) that is sometimes useful. These are $F$ such that there exists a cover $\{Y_i \rarr X\}$ in $X_{\proet}$ with $Y_i$ qcqs such that $F|_{Y_i}$ is classical and is the pullback via $\pi$ of a sheaf on the profinite set $\pi_0(Y_i)$. Here, $\pi$ refers to a map of topoi $\pi : \rmShv(Y_{\proet}) \rarr \rmShv(\pi_0(Y)_{\proet})$ discussed in \cite{BhattScholze}. A sheaf is \emph{classical} if it lies in the essential image of $\nu^* : \rmShv(X_{\et}) \rarr \rmShv(X_{\proet})$. The pullback $\nu^* : \rmShv(X_{\et}) \rarr \rmShv(X_{\proet})$ is fully faithful. Its essential image consists exactly of those sheaves $F$ with $F(U) = \rmcolim_i F(U_i)$ for any $U = \lim_i U_i$, where $i \mapsto U_i$ is a small cofiltered diagram  of affine schemes in $X_{\et}$. Moreover, if a sheaf $G \in \rmShv(X_{\proet})$ is classical when restricted to some pro-\'etale cover $\{Y_i \rarr X\}$, then $G$ is classical. By \cite[Lemma 7.3.9.]{BhattScholze}, one has $\rmLoc_X = w\rmLoc_X = \rmCov_X$, where $w\rmLoc_X$ denotes the full subcategory of locally weakly constant sheaves.
\end{rmk}

Let us gather some notions and results that play an important role in the study of the pro-\'etale topology. They were introduced in \cite[\S 2]{BhattScholze}. They are also nicely presented in \cite[Chapter 0965]{StacksProject}.

\begin{defn} 
\begin{enumerate}
  \item A spectral space $X$ is \emph{w-local} if it satisfies:
  \begin{enumerate}
      \item All open covers split, i.e., for every open cover $\{U_i \monorarr X\}$, the map $\sqcup_i U_i \rarr X$ has a section.
      \item The subspace $X^c \subset X$ of closed points is closed.
  \end{enumerate}
  \item Fix a ring $A$.
  \begin{enumerate}
  \item $A$ is \emph{w-local} if $\Spec(A)$ is w-local.
  \item $A$ is \emph{w-strictly local} if $A$ is w-local, and every faithfully flat \'etale map $A \rarr B$ has a section.
  \end{enumerate}
  \item A ring $A$ is w-contractible if every faithfully flat ind-étale map $A \rarr B$ has a section.
  \item A compact Hausdorff space is \emph{extremally disconnected} if the closure of every open is open.
\end{enumerate}
\end{defn}

\begin{fact}
  \begin{enumerate}
    \item A spectral space $X$ is w-local if and only if $X^c \subset X$ is closed, and every connected component of $X$ has a unique closed point. For such $X$, the composition $X^c \rarr X \rarr \pi_0(X)$ is a homeomorphism.
    \item A w-local ring $A$ is w-strictly local if and only if all local rings of $A$ at closed points are strictly henselian.
    \item Any ring $A$ admits an ind-\'etale faithfully flat map $A \rarr A'$ with $A'$ w-strictly local.
    \item (\cite{Gleason}) Extremally disconnected spaces are exactly the projective objects in the category of all compact Hausdorff spaces, i.e., those $X$ for which every continuous surjection $Y \rarr X$ splits.
    \item A w-strictly local ring $A$ is w-contractible if and only if $\pi_0(\Spec(A))$ is extremally disconnected.
    \item For any ring $A$, there is an ind-\'etale faithfully flat $A$-algebra $A'$ with $A'$ w-contractible.
  \end{enumerate}
\end{fact}

\addtocontents{toc}{\SkipTocEntry}
\subsection*{Recollection of some facts from Part I}
In \cite[\S 2.2]{PartI} we defined and discussed properties of a "Noohi completion". For a (Hausdorff) topological group $G$, it is defined to be $G^{\rmNoohi} = \Aut(F_G)$. It is a Noohi group and has the following properties:
\begin{itemize}
  \item any (continuous) morphism from $G$ to a Noohi group factorizes through the natural map $\alpha_G : G \rarr G^{\rmNoohi}$;
  \item $F_G$ induces an equivalence $\tilde{F}_G : G - \rmSets \rarr G^{\rmNoohi} - \rmSets$ and $\alpha_G^* \circ \tF_G \simeq \id$.
\end{itemize}

Let us now recall a part of the dictionary between statements regarding exactness of sequences of Noohi groups and statements about the induced maps on the categories of $G-\rmSets$. Before we start, recall (\cite[Defn. 2.29]{PartI}) that the "thick closure" $\bbH$ of a subgroup $H$ of a topological group $G$ is defined to be the intersection of all open subgroups of $G$ containing $H$, i.e. $\bbH:=\bigcap_{H \subset U < G \textrm{, $U$ open}}U$. If a subgroup satisfies $H = \bbH$ we will call it thickly closed in $G$.

\begin{proposition}\label{dictionary}(part of \cite[Prop. 2.37]{PartI})
    Let $G'' \stackrel{h'}\rightarrow G' \stackrel{h}\rightarrow G$ be maps between Noohi groups and $\mathcal{C''} \stackrel{H'}\leftarrow \mathcal{C'} \stackrel{H}\leftarrow \mathcal{C}$ the corresponding functors between the infinite Galois categories, i.e. $\calC = G - \rmSets$, $H = h^*$ and so on. Then the following hold:
  
\begin{enumerate}[label={(\arabic*)}]
\item \label{denseimageequivalentconditions} The following are equivalent
  \begin{enumerate}
  \item The morphism $h:G' \rarr G$ has dense image;
  \item The functor $H$ maps connected objects to connected objects.
  \end{enumerate}
  
\item $h'(G'') \subset \Ker(h)$ if and only if the composition $H' \circ H$ maps any object to a completely decomposed object, i.e. to a (possibly infinite) disjoint union of final objects.
  
\item \label{dictionary-kernel} Assume that $h'(G'') \subset \Ker(h)$ and that $h:G' \rarr G$ has dense image. Then the following conditions are equivalent:
  
\begin{enumerate}
  \item $\overline{\overline{\Image(h')}}=\Ker(h)$ and the induced map $(G'/\ker(h))^{\rmNoohi} \rarr G$ is an isomorphism;
  \item for any connected $Y \in \calC'$ such that $H'(Y)$ contains a final object of $\calC''$, $Y$ is in the essential image of $H$.
  \end{enumerate}
  \end{enumerate}
Here, being connected is equivalent to the action of the group being transitive. 
\end{proposition}

This dictionary suggests that we should work with weaker notions of exactness that translate into a statement about $G-\rmSets$. Recall from \cite[Defn. 2.38]{PartI}.
\begin{defn}\label{definition-weakly-exact}
  Let $G'' \stackrel{h'}{\rarr} G' \stackrel{h}{\rarr} G \rarr 1$ be a sequence of topological groups such that $\rmim(h') \subset \ker(h)$. Then we will say that the sequence is
  \begin{enumerate}[label={(\arabic*)}]
      \item \emph{nearly exact on the right} if $h$ has dense image,
      \item \emph{nearly exact in the middle} if $\overline{\overline{\rmim(h')}}= \ker(h)$, i.e. the thick closure of the image of $h'$ in $G'$ is equal to the kernel of $h$,
      \item \emph{nearly exact} if it is both nearly exact on the right and nearly exact in the middle.
  \end{enumerate}
  \end{defn}

Let us finish this section by mentioning two facts about the geometric coverings. The first one is stated in \cite[Prop. 3.12]{PartI} and \cite[Prop. 1.16]{ElenaPaper} and relies on the results of \cite{Rydh}.
\begin{fact}\label{properdescent}
Let $g: S' \rightarrow S$ be a proper, surjective morphism of finite presentation, then $g$ is a morphism of effective descent for geometric coverings.
\end{fact}
The following fact is a part of \cite[Thm. 4.4]{PartI}.
\begin{fact}\label{arithtogeom}
Let $X$ be a geometrically connected scheme of finite type over a field $k$. Let $\bk$ be an algebraic closure of $k$ and $\bx$ be a $\Spec(\bk)$-point on $X$. Then the induced map $\pipet(X,\bx) \rarr \Gal_k$ is open and surjective.
\end{fact}

\section{Homotopy exact sequence over a general base}
\subsection{Statement of the main result, infinite Stein factorization and some examples}
Let us state the aim of this section. As we will soon see, the main ingredient of the proof will be the construction of the "infinite Stein factorization" for geometric coverings of Thm. \ref{steinforgeomcov} below.
\begin{theorem}\label{homotopy-exact-general-base}
Let $f : X \rightarrow S$ be a flat proper morphism of finite presentation whose geometric fibres are connected and reduced. Assume that $S$ is Nagata and connected. Let $\bs$ be a geometric point of $S$ and let $\bx$ be a geometric point on $X_{\bs}$.  Then the sequence induced on the pro-\'etale fundamental groups
\begin{displaymath}
\pipet(X_{\bs},\bx) \rightarrow \pipet(X,\bx) \rightarrow \pipet(S,\bs) \rightarrow 1
\end{displaymath}
is nearly exact (see Defn. \ref{definition-weakly-exact}).

Moreover, the induced map
\begin{displaymath}
\Big(\pipet(X,\bx)/\overline{\overline{\rmim(\pipet(X_{\bs},\bx))}}\Big)^{\rmNoohi} \rarr \pipet(S,\bs)
\end{displaymath}
is a homeomorphism. Here, $\bbH$ denotes the "thick closure" defined earlier.
\end{theorem}
We will usually omit the base points.

Although it is not a part of the statement, it is true that the group $\pipet(X_{\bs})$ does not depend on the underlying field of the geometric point of $\bs$. This is by properness of $f$ and \cite[Prop. 3.31.]{PartI}.
\begin{rmk}\label{remark-that-only-nearly-exact} An example that "nearly exact" is needed in the statement (i.e. we need thick closures): let $R$ be a complete dvr with algebraically closed residue field $k$. Denote by $K$ the field of fractions of $R$. Let $X$ be a normal scheme proper over $R$ such that $X_K$ is an elliptic curve and $X_k$ is a nodal curve. For example, take $R = \bbC[[t]]$ and $X = \{ZY^2 = W^3 + ZW^2 + Z^3t\} \subset \bbP_R^2$. It is a normal scheme. We then have $\pipet(X_k) = \bbZ$, $\pipet(X) = \piet(X) = \piet(X_k) = \widehat{\bbZ}$ and $\pipet(\Spec(R))=\piet(\Spec(R))=0$. We used \cite[Exp. X, Théorème 2.1]{SGA1} to identify $\piet(X) = \piet(X_k)$.
\end{rmk}

\begin{rmk}
If $G$ is a Noohi group and $H = \overline{H} \lhd G$ is its closed normal subgroup, then the quotient group $G/H$ still has the property that open subgroups form a basis of open neighbourhoods of the identity. This is because the image of some basis of neighbourhoods of the identity of a Hausdorff group via quotient map is a basis of neighbourhoods of $1$ in the quotient group (see \cite[\S3.2, Prop. 17]{BourbakiGT}). However, the quotient $G/H$ might still fail to be Ra{\u \i}kov complete and thus Noohi. The counterexample can be obtained by applying (the proof of) \cite[Prop. 11.1]{UniformStructures}) (which gives a way of producing many examples of Ra{\u \i}kov complete groups with non-complete quotients) to a non-complete abelian group whose open subgroups form a basis of neigbourhoods of $1$, e.g. $(\bbZ_{(p)},+)$ with $p$-adic topology.
The basis of open neighbourhoods of the constructed group is uncountable and so the group is non-metrizable. This is necessary, as quotients of metrizable complete groups remain complete, see \cite[\S IX.3.1, Prop. 4]{BourbakiGT}.
\end{rmk}

\begin{proof}(of near exactness on the right in Thm. \ref{homotopy-exact-general-base}, i.e. the image of $\pipet(X) \rarr \pipet(S)$ is dense): by Prop. \ref{dictionary}, it is enough to check that the pullback of a connected geometric covering of $S$ remains connected. This follows directly from Lemma \ref{morphism-as-in-hes-over-connected-is-conn} below. Let us give a slightly different proof. Geometrically connected and reduced fibres imply that $f_*(\calO_X)=\calO_S$ (see \cite[Exc. 28.1.H]{Vakil}). Let $u:U \rightarrow S$ be a connected geometric covering. As $X \rarr S$ is in particular quasi-compact and separated and $U \rarr S$ is flat, \cite[Thm. 24.2.8]{Vakil} applies (i.e. "cohomology commutes with flat base change"), and we get that in this situation $u^*f_*\shfF=f_{U*}u^*_X\shfF$ for a quasicoherent sheaf on $X$. Applying this to $\calO_X$ and using our assumption, we get $f_{U*}\calO_{X \times_S U}=\calO_U$. Now, if $X\times_S U$ were disconnected, $\calO_{X\times_SU}$ could be written as a product of two sheaves of algebras and the same would be true for $f_{U*}\calO_{X \times_S U}=\calO_U$, which would contradict connectedness of $U$.
\end{proof}

Let us remark that if $S$ is normal, we get actual surjectivity on the right, see Lm. \ref{open-on-pi1-X-S-normal}.
\begin{defn}
To make the statements shorter, we will call a morphism of schemes $f:X \rarr S$ a "\uline{morphism as in h.e.s.}" if $S$ is connected and $f$ is flat proper of finite presentation with geometrically connected and geometrically reduced fibres.
\end{defn}

Let us state the main result. Its proof and the preparatory lemmas will occupy the subsections \S \ref{section:preliminary} and \S\ref{section:proof}. It will imply the near exactness in the middle in Thm. \ref{homotopy-exact-general-base}.

\begin{theorem}\label{steinforgeomcov} Let $S$ be a Nagata scheme. Let $X \rarr S$ be as in h.e.s. and let $Y \in \rmCov_X$ be connected. Then there exists a connected $T \in \rmCov_S$ and a morphism $g:Y \rarr T$ over $X \rarr S$ such that $g$ has geometrically connected fibres. 

Moreover, for any two $T_1$, $T_2$ and maps $g_i: Y \rarr T_i$, $i=1,2$, as in the statement, there exists a unique isomorphism $\phi:T_1 \simeq T_2$ in $\rmCov_S$ making the diagram
\begin{center}
\begin{tikzpicture}
\matrix(a)[matrix of math nodes,
row sep=1em, column sep=4.5em,
text height=0.75ex, text depth=0.25ex]
{  &  T_1  \\ Y &   \\  & T_2 \\};

\path[->] (a-2-1) edge (a-1-2);
\path[->] (a-2-1) edge (a-3-2);
\path[->] (a-1-2) edge node[right]{$\phi$} (a-3-2);
\end{tikzpicture}
\end{center}
commute.
\end{theorem}

\begin{defn}
In the situation of Thm. \ref{steinforgeomcov}, we will refer to the scheme $T \in \rmCov_S$ as the \emph{infinite Stein factorization} of $Y$.
\end{defn}
In the case when $Y \in \rmCov_X$ is a finite \'etale cover, the "infinite Stein factorization" coincides with the usual Stein factorization of the map $Y \rarr S$. See \cite[Lemma 0BUN]{StacksProject} or \cite[Exp. X, Proposition 1.2]{SGA1}.

\begin{proposition}
Assume that Theorem \ref{steinforgeomcov} holds. Then Theorem \ref{homotopy-exact-general-base} holds.
\end{proposition}
\begin{proof}
We have already checked above that the image of the map $\pipet(X_{\bs}) \rarr \pipet(X)$ is dense. We have to prove the remaining statements. 

Let $Y \in \rmCov_X$ be connected and such that $Y_{\bs}$ has a section $\sigma: X_{\bs} \rarr Y_{\bs}$. Let $T \in \rmCov_S$ be the "infinite Stein factorization" of $Y$ over $S$ constructed in Thm. \ref{steinforgeomcov}. The section $\sigma$ gives that $Y_{\bs}$ contains a copy $X_{\bs}'$ of $X_{\bs}$ as a connected clopen subset. Observe that $T_{\bs} \simeq \sqcup_t \bs$, and so $T_{X_{\bs}}\simeq \sqcup_t X_{\bs}$. As $Y_{\bs} \rarr T_{\bs}$ has connected fibres, one easily checks that that $X_{\bs}'$ equals one of the fibres and the restriction of $Y_{\bs} \rarr T_{X_{\bs}}$ to $X_{\bs}'$ is an isomorphism. Thus, one of the geometric fibres of $Y_{\bs} \rarr T_{X_{\bs}}$ is a singleton and so the same holds for $Y \rarr T_X$. As $Y$ and $T_X$ are connected geometric coverings of $X$, we conclude that $Y \rarr T_X$ is an isomorphism. By Prop. \ref{dictionary}, the proof is finished.
\end{proof}

Let us also recall some facts about Nagata schemes. Firstly, Nagata schemes are locally Noetherian (by definition).
\begin{fact}(\cite[Lemma 035S]{StacksProject})
Let $X$ be a Nagata scheme. Then the normalization $\nu : X^\nu \rarr X$ is a finite morphism.
\end{fact}

\begin{fact} The spectra of the following rings are Nagata schemes: fields, Noetherian complete local rings, Dedekind rings of characteristic zero.
Moreover, any scheme locally of finite type over a Nagata scheme is Nagata.
\end{fact}
\begin{proof}
See \cite[Tag 035A]{StacksProject} and \cite[Tag 035B]{StacksProject}.
\end{proof}

\subsubsection*{Example in the case $S$ is normal}
As an example, let us apply the stated homotopy exact sequence in the case $S$ - normal. The direct analogue of the following result holds for the \'etale fundamental groups and can be checked using the usual homotopy exact sequence and diagram chasing. The point of the following proof is to show that we can redo this proof even if we have only near exactness. 
\begin{corollary}\label{application-of-hes-when-S-normal}
Let $f:X \rarr S$ be as in h.e.s. Assume $S$ to be normal and locally noetherian. Let $\xi$ be its generic point. Then the induced morphism
\begin{displaymath}
\alpha:\pipet(X_\xi) \rarr \pipet(X)
\end{displaymath}
has dense image.
\end{corollary}
\begin{proof}
Let $\bxi$ be a geometric point over $\xi$. Denote $K=\kappa(\xi)$. Applying Thm. \ref{homotopy-exact-general-base} (see also Rmk. \ref{rmk-about-Nagata-assumption}), we have the following diagram
\begin{center}
\begin{tikzpicture}
\matrix(a)[matrix of math nodes,
row sep=1.5em, column sep=2em,
text height=2ex, text depth=0.25ex]
{\pipet(X_{\bxi}) & \pipet(X_\xi) &  \Gal_K & 1\\  \pipet(X_{\bxi}) & \pipet(X) &  \pipet(S) & 1\\};

\path[->] (a-1-1) edge (a-1-2);
\path[->] (a-1-2) edge (a-1-3);
\path[->] (a-1-3) edge (a-1-4);

\draw[double distance = 1.5pt](a-1-1) -- (a-2-1);
\path[->] (a-1-2) edge node[left]{$\alpha$} (a-2-2);
\path[->] (a-1-3) edge (a-2-3);

\path[->] (a-2-1) edge (a-2-2);
\path[->] (a-2-2) edge node[above]{$\pipet(f)$} (a-2-3);
\path[->] (a-2-3) edge (a-2-4);
\end{tikzpicture}
\end{center}
with nearly exact rows. As $S$ is normal, we have $\pipet(S) = \piet(S)$ and the map $\Gal_K \rarr \piet(S)$ is surjective by \cite[Prop. 0BQM]{StacksProject}.

Let $U \subset \pipet(X)$ be an open subgroup and let $g \in \pipet(X)$. We need to show, that $\rmim(\alpha)\cap gU \neq \emptyset$. By Thm. \ref{homotopy-exact-general-base}, $\big(\pipet(X)/\ker(\pipet(f)) \big)^{\rmNoohi}  \simeq 
\pipet(S)$ and thus there is $h \in \rmim(\alpha)\cap (\ker(\pipet(f)) \cdot gU)$. It can be also seen more directly, using Lm. \ref{open-on-pi1-X-S-normal}: it implies that the morphism $\pipet(f)$ is open. It is also surjective and thus it is a quotient map and $\pipet(X)/\ker(\pipet(f)) \simeq \pipet(S)$ as topological groups (thus in fact we do not need to pass to Noohi completions in this case). So, from the diagram, we know that $\alpha$ is surjective modulo $\ker(\pipet(f))$. But from the near exactness of $\pipet(X_{\bxi}) \rarr \pipet(X) \rarr \pipet(S)$, we have that $\ker(\pipet(f)) \cdot gU = \rmim(\pipet(X_{\bxi}))\cdot gU$ and so $h \in \rmim(\pipet(X_{\bxi}))\cdot gU$ implies that $h = xgu$ with $x \in \rmim(\pipet(X_{\bxi}))$, $u \in U$ and so $x^{-1}h \in \rmim(\alpha)\cap gU$.
\end{proof}

\begin{rmk}
Using Thm. \ref{steinforgeomcov} directly, one can give a short alternative proof of the above Corollary. Indeed, let $Y \in \rmCov_X$ be connected. Let $T \in \rmCov_S$ be the scheme obtained by applying Thm. \ref{steinforgeomcov} (see also Rmk. \ref{rmk-about-Nagata-assumption}). It is connected and thus finite (as $\pipet(S) = \piet(S)$). $T_\xi$ is connected (by the surjectivity of $\Gal_K \rarr \piet(S)$). The morphism $Y \rarr T$ (and so also $Y_\xi \rarr T_\xi$) has geometrically connected fibres. As $Y_\xi \rarr T_\xi$ is open and surjective, this implies that $Y_\xi$ is connected, as desired.
\end{rmk}

\begin{rmk}\label{rmk-about-Nagata-assumption}
Formally, in the proof of Cor. \ref{application-of-hes-when-S-normal} we should assume $S$ to be normal \underline{and Nagata} to apply Thm. \ref{homotopy-exact-general-base} or Thm. \ref{steinforgeomcov}, but examining the proofs we see that we only use that a Nagata scheme is locally noetherian and its normalization $S^\nu \rarr S$ is finite.
\end{rmk}

The following lemma is independent, i.e. we do not assume Thm. \ref{homotopy-exact-general-base} in the proof.

\begin{lemma}\label{open-on-pi1-X-S-normal}
Let $S$ be a locally noetherian normal domain and $\xi$ its generic point. Let $f:X \rarr S$ be a quasi-separated morphism  of finite type. Assume $X$ is connected and the fibre $X_{\xi}$ is geometrically connected. Then the induced morphism
\begin{displaymath}
\pipet(f):\pipet(X) \rarr \pipet(S)=\piet(S)
\end{displaymath}
is open and surjective.
\end{lemma}
\begin{proof}
We have a following diagram
\begin{center}
\begin{tikzpicture}
\matrix(a)[matrix of math nodes,
row sep=1.5em, column sep=2em,
text height=2ex, text depth=0.25ex]
{\pipet(X_\xi) &  \Gal_K \\ \pipet(X) &  \piet(S)\\};

\path[->>] (a-1-1) edge (a-1-2);

\path[->] (a-1-1) edge (a-2-1);
\path[->>] (a-1-2) edge (a-2-2);

\path[->] (a-2-1) edge (a-2-2);.

\end{tikzpicture}
\end{center}
As $\pipet(X)$ is Noohi, it is enough to show that the image of any open subgroup $U \subset \pipet(X)$ is open. Fix such $U$. Let $V$ be the preimage of $U$ in $\pipet(X_\xi)$. Then the image $\pipet(f)(U)$ contains the image of $V$ via $\pipet(X_\xi) \rarr \Gal_K \rarr \piet(S)$. But both $\pipet(X_\xi) \rarr \Gal_K$ and $\Gal_K \rarr \piet(S)$ are open (the first one by Fact \ref{arithtogeom} and the second one follows from Fact \ref{open-mapping-top-grps}) and thus $\pipet(f)(U)$ contains an open subset and so is open as desired.
\end{proof}

\begin{fact}\label{open-mapping-top-grps} Let $f:G \rarr H$ be a surjective morphism of Hausdorff topological groups. Assume $G$ compact. Then $f$ is open.
\end{fact}
\begin{proof}
As stated, this fact can be easily checked by hand. This is also a special case of a more general "Open mapping theorem", see \cite[Thm. 7.2.8]{Dikranjan}.
\end{proof}

\subsection{Preliminary results on connected components of schemes}\label{section:preliminary}

\subsubsection*{$\pi_0$ of an qcqs scheme}
We found Appendix A of \cite{Schroer} to be a handy reference for dealing with connected components of fpqc schemes. We include two useful statements below.

Let $X$ be a topological space and $a \in X$ a point. The \emph{connected
component} containing $a$, denoted $C_a$, is the union of all connected subsets containing $a$. This is
the  largest  connected subset containing $a$, and it is closed. In contrast, the \emph{quasicomponent} $Q = Q_a$ is defined as the intersection of all clopen neighborhoods of a, which is also closed. We list some handy facts below.

The following result is stated as \cite[Lemma A.1]{Schroer}. As remarked there, it is due to Ferrand in the affine case, and Lazard in the general case (see \cite[Prop. 6.1]{Lazard}, \cite[Cor. 8.5]{Lazard}).
\begin{lemma}\label{connected-comps-and-quasicomps}
Let $X$ be a qcqs scheme and $a \in X$ be a point. Then we have an equality $C_a = Q_a$ between the connected component and the quasicomponent containing $a$.
\end{lemma}
By \cite[Tag 0900]{StacksProject}, each quasi-compact space $X$ satisfying the assertion of the lemma above has a profinite set of connected components $\pi_0(X)$.
\begin{corollary}\label{pi0-of-qcqs-is-profinite}
Let $X$ be a qcqs scheme. Then $\pi_0(X)$ is profinite. 
\end{corollary} 

Let us mention a lemma on pro-\'etale covers that fits the discussion.
\begin{lemma}\label{conntected-comps-of-w-strictly-local}
Let $S$ be an affine scheme and $\tS \rarr S$ a pro-\'etale cover by a w-strictly local affine scheme. Then
\begin{enumerate}
\item Each connected component of $\tS$ is the strict henselization of the local ring at a certain point of $S$.
\item If $S$ is moreover noetherian, connected and normal and $\eta$ denotes its generic point, then each connected component $c \subset \tS$ is noetherian, normal and its generic point is the unique point of $c$ lying over $\eta$.
The map 
\begin{displaymath}
\pi_0(\tS_\eta) \rarr \pi_0(\tS)
\end{displaymath}
is a homeomorphism.
\end{enumerate}
\end{lemma}
\begin{proof}
The scheme $\tS$ is pro-étale over $S$ and $c$ is a pro-Zariski localization of $\tS$ (this is because a connected component of a qcqs scheme can be seen as an inverse limit of clopen subschemes containing it, see Lm. \ref{connected-comps-and-quasicomps}).
In particular, $c \rarr S$ is weakly \'etale. By \cite[Lm. 094Z]{StacksProject} (it relies on the theorem by Olivier, in our case it is probably an overkill, as the morphism $c \rarr S$ is actually pro-\'etale), $c \rarr S$ induces an isomorphism on the strict henselizations of the local rings. As $\tS \rarr S$ is w-strictly local, $c$ is itself a strictly henselian local ring. It follows that $c$ is a strict henselization of a local ring at some point of $S$. So, $c$ is noetherian (strict henselization of a noetherian ring is noetherian). Being weakly \'etale over a normal scheme, $c$ is normal (see \cite[Tag 0950]{StacksProject}). The scheme $c$ is local, noetherian and normal, thus integral. By \cite[Cor. 18.8.14]{EGAIV4}, the fibre $c_\eta$ contains only one point - the generic point of $c$. We are using here that the associated primes of a reduced noetherian ring are precisely the generic points of the irreducible components (see \cite[Lm. 0EME]{StacksProject} and \cite[Lm. 05AR]{StacksProject}). It follows that $\pi_0(\tS_\eta) \rarr \pi_0(\tS)$ is a (continuous) bijection. As $\tS$ is qcqs, $\pi_0(\tS)$ is compact (see Cor. \ref{pi0-of-qcqs-is-profinite}) and so $\pi_0(\tS_\eta) \rarr \pi_0(\tS)$ is in fact a homeomorphism.
\end{proof}

However, to deal with (infinite) geometric coverings, one has to be more careful. The proof of Lm. \ref{connected-comps-and-quasicomps} relies on a useful fact on the behaviour of connected components under cofiltered limits. It is essentially \cite[Proposition 8.4.1 (ii)]{EGAIV3}, but as explained in \cite{Schroer}, in \cite{EGAIV3} the scheme is only assumed to be quasi-compact, while one needs to assume qcqs.

\begin{fact}(\cite[Proposition A.2]{Schroer}) Let $X_0$ be a quasi-compact and quasi-separated scheme, and $X_\lambda$ a filtered inverse system of affine $X_0$-schemes, and $X = \invlim_{\lambda \in \Lambda} X_\lambda$. If $X = X' \sqcup X''$ is a decomposition into disjoint open subsets, then there is some $\lambda \in \Lambda$ and a decomposition $X_\lambda = X_\lambda' \sqcup X_\lambda''$ into disjoint open subsets so that $X' , X'' \subset X$ are the respective preimages.
\end{fact}
As explained below the proof of \cite[Proposition A.2]{Schroer}, both assumptions, quasi-compact and quasi-separated, are needed in general. All the spaces we are going to deal with will be quasi-separated. But non-finite geometric coverings are not quasi-compact and thus some extra care is needed when dealing with them. Thus, we devote some time to study connected components of (often) non-quasi-compact schemes.

\subsubsection*{Some aspects of Galois action on $\pi_0$}
The following lemma is used a couple of times throughout the text. Its proof is based on the results in \cite[Tag 0361]{StacksProject} and \cite[Tag 038D]{StacksProject}.

\begin{lemma}\label{finitegeomconncomp}
Let $X$ be a connected scheme over a field $k$ with an $l'$-rational point with $l'/k$ a finite field extension. Then $\pi_0(X_{k^{\rmsep}})$ is finite, the $\Gal_k$ action on $\pi_0(X_{k^{\rmsep}})$ is continuous and there exists a finite separable extension $l/k$ such that the induced map $\pi_0(X_{k^\rmsep}) \rarr \pi_0(X_l)$ is a bijection. Moreover, there exists the smallest field (contained in $k^{\rmsep}$) with this property and it is Galois over $k$.
\end{lemma}
\begin{proof}
This is \cite[Lemma 4.3]{PartI}.
\end{proof}

\subsubsection*{Connected components, fibres and geometric coverings}
\begin{lemma}\label{irr-comp-of-cover-qc}
Let $X$ be a topologically noetherian scheme and $Y \rarr X$ be in $\rmCov_X$. Let $Z$ be an irreducible component of $Y$. Then $Z$ is quasi-compact.
\end{lemma}

\begin{proof}
The image of $Z$ in $X$ sits in an irreducible component of $X$. We can base-change the situation to that component and assume that $X$ is irreducible. Let $\eta \in Z \subset Y$ be the generic point of $Z$. Let $\tX \rarr X$ be a cover in $X_{\proet}$ by a qcqs scheme such that $\tY = Y \times_X \tX$ represents a constant sheaf, i.e. $\tY \simeq \sqcup_{i \in I} \tX$, where the indexing set $I$ is possibly infinite. The morphism $\tX \rarr X$ is qcqs. Thus, the same is true for $\tY \rarr Y$ and in turn the preimage $\tilde{E}$ of $\eta$ in $\tY$ is quasi-compact. So there is a finite subset $I' \subset I$ such that $\tilde{E} \subset \sqcup_{i \in I'} \tX \subset \tY$. Let $\tilde{Z}$ be the preimage of $Z$ in $\tY$. Any point of $Z$ generalizes to $\eta$ and so, by flatness of $\tY \rarr Y$, the going-down property implies that any point of $\tilde{Z}$ generalizes to a point in $\tilde{E}$. It follows that the closure of $\tilde{E}$ in $\tY$ contains $\tilde{Z}$. But this closure is contained in $\sqcup_{i \in I'} \tX \subset \tY$. This last set is quasi-compact. As $\tZ$ is closed in $\tY$, it is quasi-compact as well. As $\tZ \rarr Z$ is surjective, $Z$ is quasi-compact, as desired.
\end{proof}

\begin{rmk}
An alternative proof of the last lemma can be given if the normalization $X^\nu$ of $X$ is topologically noetherian. Let $X^\nu \rarr X$ be the normalization map. Then the base-change $Y\times_X X^\nu$ is the normalization $Y^\nu$ of $Y$ (see \cite[Lm. 03GV]{StacksProject}) and we have a diagram
\begin{center}
\begin{tikzpicture}
\matrix(a)[matrix of math nodes,
row sep=1.5em, column sep=2em,
text height=1ex, text depth=0.25ex]
{Y^\nu &  Y \\  X^\nu &  X. \\};

\path[->] (a-1-1) edge (a-2-1);
\path[->] (a-1-1) edge  (a-1-2);
\path[->] (a-2-1) edge (a-2-2);
\path[->] (a-1-2) edge  (a-2-2);.

\end{tikzpicture}
\end{center}
Each irreducible component of $Y$ is the image of a connected component of $Y^\nu$. Thus, it is enough to show that the connected components of $Y^\nu$ are quasi-compact. As $X^\nu$ is topologically noetherian, we can apply Lm. \ref{proetale-of-normal} to get that each connected component of $Y^\nu$ is finite (\'etale) over $X^\nu$  and thus quasi-compact.
\end{rmk}

\begin{lemma}\label{wrtngassum}
Let $X$ be a connected reduced topologically noetherian scheme and let $Y \in \rmCov_X$ be connected. Then there exist open immersions $U_n \stackrel{i_n}{\rarr} Y$ and closed immersions $Z_n \stackrel{j_n}{\rarr} Y$, $n \in \bbZ_{\geq 0}$ such that:
\begin{enumerate}
\item \label{cdtnfiltr} $U_n$, $Z_n$ are of finite type over $X$, $j_n$ factorizes through $i_n$ and $i_n$ factorizes through $j_{n+1}$, i.e. we have $Z_n \rarr U_n \rarr Z_{n+1} \rarr U_{n+1} \rarr... \rarr Y$,
\item \label{cdtnsum} $\bigcup_n U_n = Y$,
\item \label{cdtnirrcomp} Each $Z_n$ is a finite union of irreducible components of $Y$.
\end{enumerate}
\end{lemma}
\begin{proof}
Observe, that as $Y$ is locally topologically noetherian and locally of finite type over $X$, it is enough to ensure that $Z_n$ and $U_n$ are quasi-compact, to obtain that they are of finite type over $X$. By Lm. \ref{irr-comp-of-cover-qc}, every irreducible component of $Y$ is quasi-compact. Let us define $Z_1$ to be any irreducible component of $Y$ and $U_1$ to be a connected quasi-compact open neighbourhood of $Z_1$ (it exists as $Z_1$ is quasi-compact and connected and $Y$ is locally topologically noetherian and so locally connected). Now, let
\begin{displaymath}
Z_2 = \bigcup_{\textrm{irr. comp. $Z$ of $Y$ s.t. $Z \cap U_1 \neq \emptyset$}} Z
\end{displaymath}
As $U_1$ is quasi-compact subset of a locally noetherian space it is noetherian and we see that the indexing set in the above sum is finite. By Lm. \ref{irr-comp-of-cover-qc}, we see that $Z_2$ is quasi-compact. $Z_2$ is a closed subset of $Y$ and we put the reduced induced structure on it. Moreover, $U_1 \subset Z_2$ and $Z_2$ is connected. We can now take $U_2$ to be a connected quasi-compact open containing $Z_2$. Repeating this procedure we produce connected schemes $Z_n$, $U_n$ satisfying (1) and (3). To check (2) we need to show that $U_\infty=\bigcup_n U_n$ is equal to $Y$. From connectedness of $Y$, it is enough to show that $U_\infty$ is clopen. It is obviously open. In particular constructible. Thus, it is enough to show that it is closed under specialization (see \cite[Tag 0542]{StacksProject}. It is stated for noetherian topological spaces but clearly locally noetherian is enough, as for any $y \in \overline{U_\infty}$ we can check whether $y \in U_\infty$ by restricting to a topologically noetherian neighbourhood $V$ of $y$ and working with the intersection $U_\infty \cap V$). Let $\xi \in U_\infty$ and assume $\xi$ specializes to a point $y \in Y$. Let $m$ be such that $\xi \in U_m \subset U_\infty$. There exist an irreducible component $Z$ of $Y$ containing $\xi$. It is closed and so $y \in Z$. But $Z \subset Z_{m+1}$ by construction. Thus, $y \in Z_{m+1} \subset U_{m+1} \subset U_{\infty}$ as desired.

\end{proof}
\begin{rmk}\label{closureqc}
One can use the above result to check that the closure of a quasi-compact subset of $Y$ remains quasi-compact.
\end{rmk}

\begin{obs}\label{globalsectionsunion}
Let $Y$ be a scheme and let $U_1 \subset U_{2} \subset U_{3} \subset \ldots \subset Y$ be an increasing sequence of open subschemes such that $\bigcup_n U_n = Y$. Then, directly from the sheaf property, it follows that
\begin{displaymath}
\Gamma(Y,\calO_Y) = \varprojlim \Gamma(U_n, \calO_{U_n}).
\end{displaymath}
\end{obs}

\begin{lemma} Let $Y$ be a reduced scheme over an algebraically closed field $k$ having a filtration $Z_0 \subset U_0 \subset Z_1 \subset U_1 \subset \ldots \subset Y$ with $U_i$ open and $Z_i$ connected and proper over $k$. Then $\Gamma(U,\calO_U) = k$
\end{lemma}\label{globalsectionsoverfield}
\begin{proof}
By Obs. \ref{globalsectionsunion}, $\Gamma(U,\calO_U) = \varprojlim \Gamma(U_n,\calO_{U_n})$. But every map $\Gamma(U_{n+1},\calO_{U_{n+1}}) \rarr \Gamma(U_n,\calO_{U_n})$ factorizes through $\Gamma(Z_{n+1},\calO_{Z_{n+1}})=k$ (from the properness of $Z_n$ and $k=\bk$). Thus, $\varprojlim \Gamma(U_n,\calO_{U_n}) = \varprojlim \Gamma(Z_{n},\calO_{Z_{n}}) = k$.
\end{proof}

The following lemma makes precise the statement that "a flat degeneration of a disconnected scheme is either disconnected or nonreduced".
\begin{lemma}\label{reducedspecialimpliesconnectedgeneric}
Let $R$ be a dvr and let $X$ be a connected scheme flat over $R$. If the special fibre $X_s$ is reduced, then the generic fibre $X_\xi$ is connected.
\end{lemma}
\begin{proof}
This is \cite[Tag 055J]{StacksProject}.
\end{proof}

The following remark is not used later and can be skipped, but gives some extra intuition.
\begin{rmk}
Let $X$ be a connected noetherian scheme such that the normalization $X^\nu \rarr X$ is finite (e.g. $X$ Nagata). Let $Y \in \rmCov_X$ be connected and $\bx$ be a geometric point on $X$. Then $Y_{\bx}$ is countable.
Indeed, using the van Kampen theorem in \cite{PartI}, we can write $\pipet(X,\bx)$ as the Noohi completion of a quotient of $*_v^{\rmtop}G_v*D$, $v$ running over a finite set, $G_v$ - profinite and $D \simeq \bbZ^{*r}$ a discrete and countable group. Fix some $v_0$. Then, with the notation as in the proof of \cite[Thm. 4.14]{PartI}, the sets $O_{v_0}^N$ are finite and $Y_{\bx} = \bigcup_{N > 0}\bigcup_{o \in O_{v_0}^N} o$, which finishes the proof, as the $G_{v_0}$-orbits $o$ are finite.
\end{rmk}

\subsubsection*{Some topology involving $\pi_0$'s of non-noetherian schemes}
\begin{lemma}\label{openonpi0}
Let $f:  W \rarr T$ be a qcqs morphism of schemes. Assume that each connected component of $T$ is locally connected (e.g. each connected component is topologically noetherian). Assume that the image of $W$ is dense in every connected component of $T$. Then the induced map $\pi_0(f) : \pi_0(W) \rarr \pi_0(T)$ is a topological quotient map.
\end{lemma}

\begin{proof}
The map $\pi_0(f)$ is surjective by the assumption of dense images. Let $U_0 \subset \pi_0(T)$. Assume that $\pi_0(f)^{-1}(U_0)$ is open. We want to show that $U_0$ is open. As the topology on $\pi_0(T)$ is the quotient topology from $T$, it is enough to show that $U=\pi^{-1}(U_0) \subset T$ is open. We have a commutative diagram
\begin{center}
\begin{tikzpicture}
\matrix(a)[matrix of math nodes,
row sep=1.5em, column sep=2em,
text height=1ex, text depth=0.25ex]
{W &  T \\  \pi_0(W) &  \pi_0(T). \\};

\path[->] (a-1-1) edge node[left]{$\pi$} (a-2-1);
\path[->] (a-1-1) edge node[above]{$f$} (a-1-2);
\path[->] (a-2-1) edge node[above]{$\pi_0(f)$} (a-2-2);
\path[->] (a-1-2) edge node[left]{$\pi$} (a-2-2);.

\end{tikzpicture}
\end{center}
Thus, $f^{-1}(U) = \pi^{-1}(\pi_0(f)^{-1}(U_0))$ is open. To prove that $U$ is open, it is enough to show that, for each affine open $V$ of $T$, the intersection $U \cap V$ is open in $V$. Fix such $V$ and denote $W_V = f^{-1}(V)$. Observe that $f|_{W_V}^{-1}(U \cap V)= f^{-1}(V) \cap W_V$ is open in $W_V$. Consider the commutative diagram of topological spaces 

\begin{center}
\begin{tikzpicture}
\matrix(a)[matrix of math nodes,
row sep=2em, column sep=3.5em,
text height=1ex, text depth=0.25ex]
{W_V & V & T \\  \pi_0(W_V) & \pi_0(V) & \pi_0(T). \\};

\path[->] (a-1-1) edge node[left]{$\pi$} (a-2-1);
\path[->] (a-1-1) edge node[above]{$f|_{W_V}$} (a-1-2);
\path[->] (a-1-2) edge node[above]{$\subset$} (a-1-3);
\path[->] (a-2-1) edge node[above]{$\pi_0(f|_{W_V})$} (a-2-2);
\path[->] (a-2-2) edge (a-2-3);
\path[->] (a-1-2) edge node[left]{$\pi$} (a-2-2);
\path[->] (a-1-3) edge node[left]{$\pi$} (a-2-3);.

\end{tikzpicture}
\end{center}
It follows from the diagram that there exists a subset $U_0' \subset \pi_0(V)$ such that $V \cap U = \pi^{-1}(U_0')$.
Moreover, as $V$ is affine and $f$ is qcqs, $\pi_0(f|_{W_V})$ is a (continuous) surjective map of compact spaces and so a quotient map. Surjectivity of $\pi_0(f|_{W_V})$ follows from the assumptions: local connectedness of connected components of $T$ implies that each connected component of $V$ is an open subset of a connected component of $T$ and by the assumption that the image of $f$ is dense in every connected component of $T$, we get the desired surjectivity. As $\pi^{-1}(\pi_0(f|_{W_V})^{-1}(U_0')) = f^{-1}(V\cap U)$ is open and both $\pi$ and $\pi_0(f|_{W_V})$ are quotient maps, we conclude that $U_0'$ is open and thus $V \cap U$ is open as desired.
\end{proof}

\begin{lemma}\label{openonpi0'}
Let $f :  W \rarr T$ be a continuous map of topological spaces. Assume that $f$ is a topological quotient map (e.g. surjective and open or surjective and closed).
Then $\pi_0(f)$ is a topological quotient map.
\end{lemma}
\begin{proof}
Let $U_0 \subset \pi_0(T)$ be such that $\pi_0(f)^{-1}(U_0)$ is open. We want to show that $U_0$ is open. It is equivalent to checking that $U=\pi^{-1}(U_0)$ is open. But $f^{-1}(U) = \pi^{-1}(\pi_0(f)^{-1}(U_0))$ and so is open. As $f$ is a quotient map, $U$ is open as well, which finishes the proof.
\end{proof}

\begin{lemma}\label{morphism-as-in-hes-over-connected-is-conn}
Let $f: X \rarr S$ be a universally open and surjective morphism of schemes (e.g. $f$ faithfully flat locally of finite presentation) with geometrically connected fibres. Then for any morphism of schemes $\tS \rarr S$, the map induced on $\pi_0$'s by the base-change of $f$ to $\tS$ $$\pi_0(f):\pi_0(\tX) \rarr \pi_0(\tS)$$ is a homeomorphism.
\end{lemma}
\begin{proof}
We can obviously assume $\tS = S$. Let $t \in \pi_0(S)$. We can see it as a closed subscheme $t \monorarr S$ and obtain $f_t: X_t \rarr t$ via base-change. Let us first show that $\pi_0(f)$ is bijective. As $f$ is surjective, $\pi_0(f)$ is surjective as well and thus we only need to show that $X_t$ is connected. But this follows easily from the fact that $t$ is connected, $f_t$ is open, surjective and has connected fibres. Thus, $\pi_0(f)$ is a continuous bijection and by Lm. \ref{openonpi0'}, it is a homeomorphism.
\end{proof}

\subsection{Proof of Theorem \ref{steinforgeomcov}}\label{section:proof}
Let us start with two results that essentially give the homotopy exact sequence in the case $S$ equal to a spectrum of a strictly henselian ring. Recall that in this case $\pipet(S)=\piet(S)=1$ by \cite[Lemma 7.3.8]{BhattScholze}.
\begin{proposition}\label{forclosedfibre}
Let $S$ be a spectrum of a strictly henselian noetherian ring, let $X \rarr S$ be proper with $X$ connected and let $Y \in \rmCov_X$ be connected. Let $\bs$ be a geometric point over the closed point $s$ of $S$. Then the geometric fibre $Y_{\bs}$ is connected. In other words, the morphism
\begin{displaymath}
\pipet(X_{\bs}) \rarr \pipet(X)
\end{displaymath}
has dense image.

As a result, for any $Y \in \rmCov_X$, the natural map $\pi_0(Y_{\bs}) \rarr \pi_0(Y)$ is a bijection (of discrete sets).
\end{proposition}
\begin{proof}
As the residue field $\kappa(s)$ is separably closed, it is enough to show that $Y_s$ is connected (\cite[Tag 0387]{StacksProject}). Assume the contrary and write $Y_s = W_1 \sqcup W_2$, where $W_1,W_2$ are clopen subsets of $Y_s$. Apply Lm. \ref{wrtngassum} to $Y$ and produce a sequence of connected quasi-compact closed $Z_n \subset Y$ such that $\bigcup_n Z_n = Y$. There exist $N \geq 0$ such that $Z_N\cap W_1 \neq \emptyset$ and $Z_N \cap W_2 \neq \emptyset$. Thus, the fibre $Z_{N,s}$ is not connected. But $Z_N$ is of finite type over $X$ and satisfies the valuative criterion of properness (as $Y \in \rmCov_X$). Thus, $Z_N$ is proper over $X$ and so also over $S$. But for such scheme it is known (by the special case of the proper base-change theorem: see e.g. \cite[Lm. 0A3S]{StacksProject} or \cite[Lm. 0A0B]{StacksProject}) that the induced map $\pi_0(Z_{N,s}) \rarr \pi_0(Z_N)$ is bijective. This is a contradiction.
\end{proof}

\begin{lemma}\label{strhens-generic}
Let $S$ be the spectrum of a strictly henselian dvr $R$ and let $X \rarr S$ be a morphism as in h.e.s. Let $\bxi$ be a geometric point over the generic point $\xi$ of $S$. Then the morphism
\begin{displaymath}
\pipet(X_{\bxi}) \rarr \pipet(X)
\end{displaymath}
has dense image. In other words, for a connected $Y \in \rmCov_X$, $Y_{\bxi}$ remains connected. As a result, for any $Y \in \rmCov_X$, the natural map $\pi_0(Y_{\bxi}) \rarr \pi_0(Y)$ is a bijection (of discrete sets). 
\end{lemma}
\begin{proof}
Denote by $s$ the closed point of $S$.
Let $Y \in \rmCov_X$ be connected. Using Prop. \ref{forclosedfibre}, we know that $Y_s$ is connected. The scheme $Y_\xi$ is locally of finite type over $K=\kappa(\xi)$. Thus, it has an $L'$-point with $L'/K$ a finite field extension. Applying Lm. \ref{finitegeomconncomp}, we get that there exists a finite separable extension $L/K$ such that $Y_{\xi,L}$ has a finite number of connected components and each of them is geometrically connected. Let $R'$ be an integral closure of $R$ in $L$. It is a local (as $R$ is henselian) algebra finite over $R$ (as the extension $L/K$ was separable). $R'$ is thus a dvr. Let $\pi'$ be its uniformizer. The field $k'=R'/\pi'$ is a finite (purely inseparable) extension of $R/\pi=k$. Thus, we can assume $k' \subset \kappa(\bs)$. Thus, we can see $\bs$ as a geometric point of $\Spec(R')$ lying over the special point $s'$. Denoting $Y' = Y \times_S \Spec(R')$, we have $Y'_{\bs} = Y_{\bs}$ and so it is reduced and connected. We conclude that the fibre $Y'_{s'}$ is connected and reduced as well (because $Y'_{\bs} \rarr Y'_{s'}$ is faithfully flat). Thus, also $Y'$ is connected (this is clear when thinking of $Y'_{s'}$ and $Y'$ as elements of $\pipet(X'_{s'})-\rmSets$ and $\pipet(X')-\rmSets$). By Lm. \ref{reducedspecialimpliesconnectedgeneric} we conclude that the generic fibre $Y_{\rmFrac(R')}$ is connected. But $Y'_{\rmFrac(R')} = Y' \times_{\Spec(R')}\Spec(L) = Y \times_{\Spec(R)}\Spec(L) = Y_{\xi,L}$. Combining it with what we observed at the beginning of the proof, we conclude that $Y_{\xi,L}$ is geometrically connected. This finishes the proof.
\end{proof}

\begin{theorem}\label{bijectiononfibres}
Let $S$ be a connected noetherian scheme. Let $X \rarr S$ be as in h.e.s. Let $\bxi$ and $\bs$ be two geometric points on $S$ with images $\xi, s \in S$. Then, for any $Y \in \rmCov_X$, there is a bijection
\begin{displaymath}
\pi_0(Y_{\bxi}) \simeq \pi_0(Y_{\bs}).
\end{displaymath}
It depends on the choice of a "path" between $\bs$ and $\bxi$, i.e. chain of maps from strictly henselian dvrs (see the proof).
When $S$ is the spectrum of a strictly henselian dvr and $\bxi$ and $\bs$ lie over the special and the generic point, the bijection is the obvious one given by combining Lm. \ref{forclosedfibre} with Lm. \ref{strhens-generic}.

\end{theorem}
\begin{proof} As $S$ is connected and noetherian, we can join $s$ with $\xi$ by a finite sequence of specializations and generizations of points on $S$ (every point lies on one of the finitely many irreducible components $Z_1, \ldots, Z_m$ of $S$ and within a fixed irreducible component every point is a specialization of the generic point. It follows that the set of points reachable via sequence of specializations and generizations from a given point is a union of some irreducible components and thus closed. There is only finitely many of such "path components", and thus it is also open). Thus, we can and do reduce to the case where $s$ is a specialization of $\xi$. By \cite[Tag 054F]{StacksProject} we can find a dvr $R$ and a morphism $\Spec(R) \rarr S$ such that the generic point of $\Spec(R)$ maps to $\xi$ and the special point maps to $s$. The strict henselization $R^{sh}$ of $R$ is a strictly henselian dvr by \cite[Tag 0AP3]{StacksProject}. Let $\xi', s'$ be the generic and the special point of $\Spec(R^{sh})$ respectively and let $\bxi'$ and $\bs'$ be some geometric points over them. By Lm. \ref{strhens-generic} and Prop. \ref{forclosedfibre}, we conclude that  $\pi_0(Y_{\bs'}) \simeq \pi_0(Y_{\Spec(R^{sh})}) \simeq \pi_0(Y_{\bxi'})$. Choosing  geometric points $\bs''$ and $\bxi''$ on $S$ such that $\bs''$ factors both through $\bs$ and $\bs'$ and $\bxi''$ factors both through $\bxi$ and $\bxi'$ and using the fact that the base-change of a connected scheme over an algebraically closed field to another algebraically closed field remains connected, we finish the proof.
\end{proof}

\begin{corollary}\label{generalstrhenscase}
Let $S$ be the spectrum a strictly henselian noetherian ring and let $X\rarr S$ be as in h.e.s. Let $\bs$ be any geometric point on $S$. Then the morphism
\begin{displaymath}
\pipet(X_{\bs}) \rarr \pipet(X)
\end{displaymath}
has dense image. In other words, for a connected $Y \in \rmCov_X$, the base-change $Y_{\bs} \in \rmCov_{X_{\bs}}$ remains connected. As a result, for any
$Y \in \rmCov_X$, the natural map $\pi_0(Y_{\bs}) \rarr \pi_0(Y)$ is a bijection (of discrete sets).
\end{corollary}
\begin{proof}
If $\bs$ lies over the closed point of $S$, then the statement was proven in Prop. \ref{forclosedfibre}. But now Thm. \ref{bijectiononfibres} tells us that the statement holds for any $\bs$.
\end{proof}

Let us also mention a technical lemma used later in the proof.
\begin{lemma}\label{lifting-automorphism-of-pi0}
Let $X$ be a compact topological space. Let $W = \sqcup_i X_i$ be a disjoint union, indexed by $i$, of copies of $X$. Let $g: W \rarr X$ be the obvious structural map. Let $g_i: X_i \subset W \rarr X$ be the structural (iso-)morphism. Let $\phi \in \rmAut_X(W)$ be an automorphism of $W$ \textbf{over $X$}. Then
\begin{enumerate}
\item there exists a decomposition of $W$ into clopen subsets $W_{ij}$ such that:
\begin{itemize}
\item $W_{ij} \subset X_i$,
\item $X_i$ is a sum of finitely many $W_{ij}$,
\item $\phi$ preserves this decomposition, i.e. for each $i,j$, $\phi(W_{ij})=W_{i'j'}$ for some $i', j'$.
\end{itemize}
\item Assume that $X=\pi_0(X')$ for some topological space $X'$. Let $W'= \sqcup_i X_i'$ be the disjoint onion of copies of $X'$. Then $\phi$ lifts uniquely to an automorphism of $W'$ over $X'$, i.e. there exist an automorphism $\phi':W' \rarr W'$ such that $\phi=\pi_0(\phi')$.
\item Assume that $X=\pi_0(X')$ for some scheme $X'$. Let $W'= \sqcup_i X_i'$ be the disjoint onion of copies of $X'$. Then $\phi$ lifts uniquely to an automorphism of $W'$ over $X'$, i.e. there exist an automorphism $\phi':W' \rarr W'$ of $X'$-schemes such that $\phi=\pi_0(\phi')$.
\end{enumerate}
\end{lemma}
\begin{proof}
We omit the proof of the first part and only comment on the other two. If $X'$ is a topological space, the first part of the Lemma produces the clopen subsets $W_{ij}$ of $W$. Observe that if $\phi(W_{ij})=W_{i'j'}$, then $\phi|_{W_{ij}}: W_{ij} \rarr W_{i'j'}$ is equal to  $g_i'^{-1}\circ g_i|_{W_{ij}}$. This means that $\phi$ can be described purely combinatorically (in terms of the indexing set $\{i\}$ and clopen decompositions of the topological space $X$). The crucial thing here is that we assume $\phi$ to respect the structural morphism $W \rarr X$.
Taking the preimages $W_{ij}'$ of $W_{ij}$ via the projection $W' \rarr \pi_0(W')=W$, we get a clopen decomposition of $W'$. The combinatorial description of $\phi$ can be lifted in an obvious (and unique) way to an automorphism of $W'$. Moreover, for $X'$ a scheme, $\phi'$ will be an automorphism of a scheme (i.e. we get a map of the structure sheaves). This is because if $\phi'$ maps a clopen subset $W_1' \subset X_1'$ homeomorphically onto a clopen subset $W_2' \subset X_2'$, then denoting $g_1': X_1' \simeq X'$,  $g_2': X_2' \simeq X'$ (the structural isomorphisms), $\phi'|_{W_1'}$ is given by $g_2^{'-1} \circ g_1'|_{W_1'}$.
\end{proof}

\subsubsection*{Proof of Theorem \ref{steinforgeomcov}:}
\begin{proof}
Let us start with the proof of uniqueness. Let $\bx$ be a geometric point on $X$ and $\bs$ its image on $S$. Let $Y \in \rmCov_X$ be connected. Given two connected $T_i \in \rmCov_S$ and maps $g_i: Y \rarr T_i$ over $X \rarr S$ that have geometrically connected fibres, we easily see that the maps $g_i$ induce bijections $b_i: \pi_0(Y_{\bs}) \rarr (T_i)_{\bs}$. Let $"\phi_{\bs}": (T_1)_{\bs} \stackrel{\simeq}{\rarr} (T_2)_{\bs}$ be the bijection given by $b_2 \circ b_1^{-1}$. We have to check, that the bijection $"\phi_{\bs}"$ is a map of $\pipet(S,\bs)-\rmSets$, i.e. that it is $\pipet(S,\bs)$-equivariant. It is easy to check that the following diagram of $\pipet(X,\bx)-\rmSets$ is commutative
\begin{center}
\begin{tikzpicture}
\matrix(a)[matrix of math nodes,
row sep=1em, column sep=4.5em,
text height=0.75ex, text depth=0.25ex]
{  &  (T_1)_{\bs}  \\ Y_{\bx} &   \\  & (T_2)_{\bs} \\};

\path[->>] (a-2-1) edge (a-1-2);
\path[->>] (a-2-1) edge (a-3-2);
\path[->] (a-1-2) edge node[right]{$"\phi_{\bs}"$} (a-3-2); .
\end{tikzpicture}
\end{center}
Thus, $"\phi_{\bs}"$ is $\pipet(X,\bx)$-equivariant and so also $\pipet(S,\bs)$-equivariant, as $\pipet(X,\bx) \rarr \pipet(S,\bs)$ has dense image (by the part of Thm. \ref{homotopy-exact-general-base} that was already proven). The equivalence of categories $\rmCov_S \simeq \pipet(S,\bs)-\rmSets$ gives us $\phi$. This finishes the proof of uniqueness. Let us proceed with the proof of existence.

Reduction of the proof of existence to the case $S$ - normal:\\
Let $S^\nu \rarr S$ be the normalization morphism. It is finite by the fact that $S$ is Nagata. Thus, it is a morphism of effective descent for $\rmCov_S$ by Fact \ref{properdescent}. Assume that the statement of the theorem holds for normal schemes. Let $Y \in \rmCov_X$ be connected. Denote by $Y'$ and $X'$ the base-changes of $Y$ and $X$ to $S^\nu$ (we do not denote it $Y^{\nu}$ and $X^{\nu}$  to avoid the confusion with the normalizations). Applying the theorem to (each one of the discrete set of connected components of) $S^{\nu}$ we obtain $T' \in \rmCov_S$ and a surjective morphism $Y' \rarr T'$ over $X' \rarr S$. The proof will be finished if we equip $Y' \rarr T'$ with a descent datum with respect to  $S^\nu \rarr S$. Let $p,q: S'_2 = S^\nu \times_S S^\nu \rarr S^\nu$ be the two projections. We get two morphisms $Y'_2 = Y \times_S S'_2 \rarr p^*T'$ and $Y'_2 = Y \times_S S'_2 \rarr q^*T'$ (over $X'_2 \rarr S'_2$) that are surjective with geometrically connected fibres. Thus, by the uniqueness part of the theorem (that we have proven over any scheme) applied to (each connected component of) $S'_2$, we get an isomorphism $\phi: p^*T' \rarr q^* T'$. Then the cocycle condition holds by applying the uniqueness statement (over $S'_3 = S^\nu \times_S S^\nu \times_S S^\nu$) again. Observe that the morphism $Y' \rarr T'$ descends as well by the following reasoning: the schemes $Y'$, $T'$ descent to $Y$ and $T$ and then we look at $Y$ and $T_X=T \times_S X \in \rmCov_X$, and see that $Y' \rarr T' \times_{S^{\nu}} X'$ descends to a morphism $Y \rarr T_X$. This finishes the reduction. Similarly, we see that the problem is Zariski local on $S$ and we can assume $S$ to be affine.

{\centering
\underline{Thus, we can and do assume that $S$ is normal and affine in the rest of the proof.}\\
}
Let $\tS \rarr S$ be a pro-\'etale cover such that $\tS$ is affine w-strictly local (as defined in \cite[Def. 2.2.1]{BhattScholze}). This is possible by \cite[Cor. 2.2.14]{BhattScholze} or \cite[Tag 097R]{StacksProject}). All local rings at closed points of $\tS$ are strictly henselian by \cite[Lm. 2.2.9]{BhattScholze}. They are equal to the strict henselizations of the local rings at corresponding (geometric) points of $S$. By \cite[Lm. 2.1.4]{BhattScholze}, if $\tS^c$ denotes the set of closed points of $\tS$, the composition map $\tS^c \rarr \tS \rarr \pi_0(\tS)$ is a homeomorphism (as $\tS$ is w-local). From this we see that each localization at a closed point of $\tS$ is equal to a connected component of $\tS$ containing this point. Let us denote by $\tX$ and $\tY$ the base-changes of $X$ and $Y$ to $\tS$. By Cor. \ref{generalstrhenscase}, if $\ts \in \tS$ is a closed point and if $c_{\ts}$ denotes the connected component of $\ts$, then $Y_{c_{\ts}}$ is a disjoint union of connected schemes and $\pi_0(Y_{c_{\ts}})$ (it is a discrete set, as $Y_{c_{\ts}}$ is locally noetherian and so the connected components are open) can be canonically identified with $\pi_0(Y_{\ts})$. The point $\ts$ can be seen as a geometric point over its image $s \in S$. By Thm. \ref{bijectiononfibres}, we have a bijection
$\pi_0(Y_{\bs}) \simeq \pi_0(Y_{\bxi})$ for any two geometric points $\bs$ and $\bxi$ on $S$. Thus, we see that $\pi_0$ of restrictions of $\tY$ to two connected components of $\tS$ can be  identified, i.e. if $c_1, c_2 \subset \tS$ are two connected components, there is a bijection of discrete sets $\pi_0(\tY_{c_1}) \simeq \pi_0(\tY_{c_2})$. Thus, as sets, we can write $\pi_0(\tY) \simeq \sqcup_{t \in \pi_0(Y_{\bs})} \pi_0(\tS)_t$ and this identification is compatible with the natural maps $\pi_0(\tY) \rarr \pi_0(\tS)$ and $\pi_0(\tS)_t \stackrel{\id}{\rarr} \pi_0(\tS)$. Here $\bs$ is some fixed (arbitrarily chosen) geometric point of $S$ and the subscript notation in $\pi_0(\tS)_t$ denotes different copies of the set $\pi_0(\tS)$ ($\pi_0(\tS)_t$ does \underline{not} denote a fibre over some $t\in S$!). As we explain below, it can be upgraded to a homeomorphism. Observe that $\pi_0(\tS)$ is a profinite set (as $\tS$ is qcqs), but usually it will not be finite.

\textbf{Claim 1}: There is a homeomorphism $\alpha: \sqcup_{t \in \pi_0(Y_{\bs})} \pi_0(\tS)_t \rarr \pi_0(\tY)$ over $\pi_0(\tS)$.\\
Proof of the claim:
Let $\eta$ be the generic point of $S$ (recall that $S$ is now assumed to be normal). By Lm. \ref{conntected-comps-of-w-strictly-local}, we can identify $\pi_0(\tS_\eta)=\pi_0(\tS)$. 
Let us show that we can identify $\pi_0(\tY_\eta)=\pi_0(\tY)$ as well. Here $\tY_\eta = \tY \times_S \eta = \tY \times_{\tS} \tS_\eta=Y_\eta\times_S \tS$.
\begin{lemma*}
The map $\pi_0(\tY_\eta) \rarr \pi_0(\tY)$ is a homeomorphism.
\end{lemma*}
To show that  $\pi_0(\tY_\eta) \rarr \pi_0(\tY)$ is bijective it is enough to look fibre by fibre over $\pi_0(\tS)$. Thus, we can fix a connected component $c \in  \pi_0(\tS)$ and base-change to $c$ (keep in mind that, as a morphism of schemes, $c \rarr S$ is (among other properties) a closed immersion. In particular, $\pi_0(\tY \times_{\tS} c)$ is equal to the preimage of $c$ under $\pi_0(\tY) \rarr \pi_0(\tS)$. Similarly for $\tY_\eta$.). The component $c$ is a strict henselization at a (geometric) point on $S$ (Lm. \ref{conntected-comps-of-w-strictly-local}) and so is noetherian. The fibre $c_\eta$ consists of a single point: the generic point of $c$ ($c$ is a normal connected scheme), let us call it $\xi$. Then the map $(\tY_\eta)_c \rarr \tY_c$ (where the subscript c denotes the base-change from $\tS$ to $c$) is equal to the embedding of the fibre over $\xi$ to $\tY_c$ (i.e the base-change of $\xi \rarr c$ to $\tY_c$). By Cor. \ref{generalstrhenscase}, $\pi_0((\tY_c)_{\bxi}) \rarr \pi_0(\tY_c)$ is a bijection. But we have a factorization $\pi_0((\tY_c)_{\bxi}) \epirarr \pi_0((\tY_c)_{\xi}) \rarr \pi_0(\tY_c)$, where the first map is surjective. It follows that $\pi_0((\tY_c)_{\xi}) \rarr \pi_0(\tY_c)$ is a bijection. But $(\tY_c)_{\xi} \simeq (\tY_c)_\eta$ and the proof of bijectivity is finished. By Lm. \ref{openonpi0}, $\pi_0(\tY_\eta) \rarr \pi_0(\tY)$ is a homeomorphism and the proof of the lemma is finished.

Thus, we can focus on understanding $\pi_0(\tY_\eta)$. The scheme $Y_\eta$ is a (possibly infinite in this proof, but see Rmk. \ref{remark-indexing-set-finite}) disjoint union of connected components belonging to $\rmCov_{X_\eta}$. These components are clopen and so $\pi_0(Y_\eta)$ and $\pi_0(\tY_\eta)$ split accordingly. We can thus restrict attention to one connected component and assume $Y_\eta$ connected in the proof of the claim. Let $c \in \pi_0(\tS_\eta)$. The component $c$ is the spectrum of a separable algebraic field extension $L_c$ of $K= \kappa(\eta)$. By Cor. \ref{generalstrhenscase} and Lm. \ref{finitegeomconncomp}, the base-change $\tY_c$ is a disjoint union of finitely many components and each of them is geometrically connected. Thus, $Y_{L_c}$ has geometrically connected components. Moreover, the number of these connected components is constant when $c$ varies (by Thm. \ref{bijectiononfibres}), let us say equal $M$. The scheme $\tS_\eta$ is pro-\'etale over $\eta$ and so a cofiltered limit $\lim S_\lambda$ of finite unions of spectra of finite separable extensions of $K$.
\begin{lemma*}
For some $\lambda_0$, there is $S_{\lambda_0} = \sqcup_i \Spec(L_i)$ with $L_i/K$ finite separable and such that the connected components of $Y_{L_i}$ are geometrically connected (or, in other words, $Y_{L_i}$ has precisely $M$ connected components).
\end{lemma*}
Indeed,  for each $S_\lambda$, let $W_\lambda \subset S_\lambda$ be the union of those $\Spec(L_i) \subset S_\lambda$ that \underline{do not} have this property. This forms a sub-inverse system of $S_\lambda$. We want to show that, for some $\lambda_0$, $W_{\lambda_0}$ is empty. Assume the contrary. The maps between $W_\lambda's$ are affine and thus $\widetilde{W} = \lim W_\lambda$ exists in the category of schemes and moreover $\widetilde{W}_{\rmtop} = \lim W_{\lambda,\rmtop}$ (\cite[Tag 0CUF]{StacksProject}), where $W_{\rmtop}$ denotes the underlying topological space of a scheme $W$. As $W_{\lambda,\rmtop}$ are finite and non-empty, the inverse limit is non-empty as well (see \cite[Lm. 086J]{StacksProject}). The image of any point $w\in \widetilde{W}$ in $\tS$ gives a point of $\tS$ which has as the residue field a separable extension $L/K$ such that $Y_L$ is a disjoint union of geometrically connected components, but $L$ can be written as a filtered colimit of fields $L_\alpha$ with $L_\alpha/K$ finite separable and such that the connected components of $Y_{L_\alpha}$ are not all geometrically connected. But $L$ must contain the smallest field $L_{\mathrm{smallest}}$ of Lm. \ref{finitegeomconncomp} and consequently (using that $L_{\mathrm{smallest}}/K$ is finite) one of $L_\alpha$ must contain $L_{\mathrm{smallest}}$ as well, which (by Lm. \ref{finitegeomconncomp} again) contradicts the fact that $Y_{L_{\alpha}}$ has a component that is not geometrically connected. Thus, we proved that there exists $\lambda_0$ such that $S_{\lambda_0} = \sqcup_{i=1}^{m_0} \Spec(L_i)$ with $L_i/K$ finite separable and such that the connected components of $Y_{L_i}$ are geometrically connected. This finishes the proof of the lemma.

Now, we have an equality $\pi_0(Y_{L_i})=\pi_0(Y_{\bar{\eta}})$. Taking the preimages of the connected components of each $Y_{L_i}$ in $\tY_\eta$, we see that
$\tY_{\eta}$ decomposes as a disjoint union of clopen subsets $\tZ_t$, parametrized by $t \in \pi_0(Y_{\bar{\eta}})$, such that each $\tZ_t$ maps surjectively onto $\tS_\eta$ and induces a continuous bijection $\pi_0(\tZ_t) \rarr \pi_0(\tS_\eta)$. More precisely, we have a diagram 
\begin{center}
\begin{tikzpicture}
\matrix(a)[matrix of math nodes,
row sep=1.5em, column sep=1em,
text height=2ex, text depth=0.25ex]
{\tY_\eta & \hspace{3mm}  & \tS_\eta & \\ \sqcup_i Y_{L_i} & \hspace{3mm} & S_{\lambda_0} & \sqcup_i \Spec(L_i) \\ Y_\eta & \hspace{3mm} & \eta & \\};

\path[->] (a-1-1) edge (a-2-1);
\path[->] (a-2-1) edge (a-3-1);
\path[->] (a-1-3) edge (a-2-3);
\path[->] (a-2-3) edge (a-3-3);
\path[->] (a-1-1) edge (a-1-3);
\path[->] (a-2-1) edge (a-2-3);
\path[->] (a-3-1) edge (a-3-3);
\draw[double distance = 1.5pt](a-2-3) -- (a-2-4);

\end{tikzpicture}
\end{center}
and we know that for each $i$, $Y_{L_i} = \sqcup_{t \in \pi_0(Y_{\bar{\eta}})}Z_{i,t}$, with $Z_{i,t} \rarr \Spec(L_i)$ geometrically connected. We define $\tZ_{i,t}$ to be the preimage of $Z_{i,t}$ in $\tY_{\eta}$ and put $\tZ_i = \sqcup_{t\in \pi_0(Y_{\bar{\eta}})} \tZ_{i,t}$. As the map $\tZ_t \subset \tY_{\eta} \rarr \tS_{\eta}$ is open, we get by Lm. \ref{morphism-as-in-hes-over-connected-is-conn} that $\pi_0(\tZ_t) \rarr \pi_0(\tS_\eta)$ is actually a homeomorphism. Thus, we get a homeomorphism $\pi_0(\tY_\eta) \simeq \underset{t \in \pi_0(Y_{\bar{\eta}})}{\sqcup} \pi_0(\tS_\eta)$ as desired.

By the last claim, there is an isomorphism $\tY \simeq \underset{t \in \pi_0(Y_{\bs})}{\sqcup} \tY_t$, i.e. $\tY$ splits as a union of clopen subsets parametrized by $t \in \pi_0(Y_{\bs})$. Define $\tT = \underset{t \in \pi_0(Y_{\bs})}{\sqcup} \tS_t$, where $\tS_t$ is a copy of $\tS$. There is an obvious morphism $\tY \rarr \tT$, which restricted to a fixed $\tY_t$ factorizes through $\tS_t$. The scheme $\tT$ is in $\rmCov_{\tS}$ and we want to show that it descends to a covering of $T$. The morphism $\tY \rarr \tT$ is surjective and has geometrically connected fibres. Indeed, to see the surjectivity, observe that by Lm. \ref{morphism-as-in-hes-over-connected-is-conn} and by the construction of $\tT$, $\pi_0(\tY)\rarr \pi_0(\tT_{\tX})$ is a homeomorphism and the connected components of $\tT_{\tX}$ are isomorphic to connected components of $\tX$ and thus noetherian. Restricting to such a component, $\tY \rarr \tT_{\tX}$ becomes a geometric covering with dense image and thus surjective, e.g. by \cite[Lm. 7.3.9]{BhattScholze}. To see the connectedness of geometric fibres, observe that, by the construction, $\pi_0(\tY) \rarr \pi_0(\tT)$ is a homeomorphism and $\tT = \pi_0(Y_{\bar{\eta}})\times \tS$. We can restrict the situation to a fixed connected component $c$ of $\tS$. Identifying $Y_{\bar{\eta}} \simeq \tY_c\times_c \bar{\eta}$ (using some lift of $\bar{\eta}$ from $S$ to $c$), we see that the $\bar{\eta}$-fibre of $\tY \rarr \tT$ is connected. But now it follows quite easily from Cor. \ref{generalstrhenscase} that in fact every geometric fibre is connected. The proof of existence of the descent datum and checking the cocycle condition follows essentially from the uniqueness of the infinite Stein factorization. We cannot, however, simply apply the uniqueness statement proven above, as $\pi_0(\tS)$ is not discrete (and we cannot simply argue by restricting to the connected components). Thus, we have to be slightly more careful. We need to equip $\tT$ with a descent datum. Denote for brevity $\tS_2 = \tS\times_S\tS$, $\tY_2 = \tY\times_Y\tY$, $\tX_2 = \tX\times_X\tX$ and let $p,q:\tS_2 \rarr \tS$ be the canonical projections. We need to define an isomorphism over $\tS_2$ between the two base-changes $p^*\tT$ and $q^*\tT$. We have diagrams
\begin{center}
\begin{tikzpicture}
\matrix(a)[matrix of math nodes,
row sep=1.5em, column sep=2em,
text height=2ex, text depth=0.25ex]
{\tY_2 & p^*\tT &  \tS_2 &  & \tY_2 & q^*\tT &  \tS_2\\  \tY & \tT &  \tS &  & \tY & \tT &  \tS\\};

\path[->] (a-1-1) edge node[left]{$p$} (a-2-1);
\path[->] (a-1-1) edge node[above]{$\alpha$} (a-1-2);
\path[->] (a-2-1) edge (a-2-2);
\path[->] (a-1-2) edge  (a-2-2);

\path[->] (a-1-2) edge (a-1-3);
\path[->] (a-2-2) edge (a-2-3);
\path[->] (a-1-3) edge node[left]{$p$} (a-2-3);

\path[->] (a-1-5) edge node[left]{$q$} (a-2-5);
\path[->] (a-1-5) edge node[above]{$\beta$} (a-1-6);
\path[->] (a-2-5) edge (a-2-6);
\path[->] (a-1-6) edge (a-2-6);

\path[->] (a-1-6) edge (a-1-7);
\path[->] (a-2-6) edge (a-2-7);
\path[->] (a-1-7) edge node[left]{$q$} (a-2-7);

\end{tikzpicture}
\end{center}
with all squares Cartesian. We claim that: $\alpha$ and $\beta$ induce homeomorphisms on $\pi_0$'s. Indeed, $\tY \rarr \tT$ is universally open, surjective  with geometrically connected fibres (we are using \cite[Thm. 2.4.6]{EGAIV2} here to check openness. To check that $\tY \rarr \tT$ is flat and locally of finite presentation, it is enough to these properties for $\tY \rarr \tT_{\tX}=\tT\times_{\tS}\tX$, as $\tT_{\tX} \rarr \tT$ has the desired properties. But $\tY \rarr \tT_{\tX}$ is a morphism of \'etale $\tX$-schemes, so the properties follows. The surjectivity was proven above). Thus, the same is true for $\alpha$ and $\beta$. These assumptions imply that $\pi_0(\alpha)$ and $\pi_0(\beta)$ are continuous bijections and in fact homeomorphisms by Lm. \ref{morphism-as-in-hes-over-connected-is-conn}.

From this claim we obtain a homeomorphism $\phi_0 = \pi_0(\beta)\circ \pi_0(\alpha)^{-1}: \pi_0(p^*\tT) \rarr \pi_0(q^*\tT)$ over $\pi_0(\tS_2)$.

\textbf{Claim 2}: $\phi_0$ lifts uniquely to an isomorphism $\phi: p^*\tT \rarr q^*\tT$ over $\tS_2$.\\
Proof of the claim: $p^*\tT$ and $q^*\tT$ are both isomorphic over $\tS_2$ to a disjoin union $\sqcup_{t \in \pi_0(Y_{\bs})} \tS_2$. Fixing these isomorphisms, we can view $\phi_0$ as a homeomorphism of $\sqcup_{t \in \pi_0(Y_{\bs})} \pi_0(\tS_2)$ with itself and we want to show, that it lifts to an isomorphism of $\sqcup_{t \in \pi_0(Y_{\bs})} \tS_2$. This follows by Lm. \ref{lifting-automorphism-of-pi0}, as $\pi_0(\tS_2)$ is compact and $\phi_0$ is over the base $\pi_0(\tS_2)$.

We need to show that $\phi$ satisfies the cocycle condition. Let $\tS_3=\tS\times_S \tS\times_S\tS$ and analogously for $\tY$. For $i\in \{1,2,3\}$ let $p_i: \tS_3 \rarr \tS_2$ be the projection forgetting the $i$-th factor and for $i \neq j$ in $\{1,2,3\}$ denote $p_{ij}:\tS_3 \rarr \tS$ the projection forgetting the $i$-th and $j$-th factors. As in the case of double products, there  are morphisms $\tY_3 \rarr p_{ij}^*\tT$ fitting into suitable Cartesian diagrams. Denoting by $a$ and $b$ the morphisms $\tY_3 \rarr p_{13}^*\tT$ and $\tY_3 \rarr p_{12}^*\tT$ respectively, we have a commutative diagram 
\begin{center}
\begin{tikzpicture}
\matrix(a)[matrix of math nodes,
row sep=0.2em, column sep=1.7em,
text height=2ex, text depth=0.25ex]
{     &   {}  &             &  p_1^*(p^*\tT)=p_{13}^*\tT &  {}   &      \\
\tY_3 &     &             &              &     & \tS_3\\
      &     & p_1^*(q^*\tT)=p_{12}^*\tT &              &     &      \\
      &     &             &     {}         &     &      \\
      &     &             &              &     &      \\
      &     &             &  p^*\tT &     &      \\
\tY_2 &     &             &              &     & \tS_2\\
      &     & q^*\tT &              &     &      \\};

\path[->] (a-2-1) edge node[above left]{$a$} (a-1-4);
\path[->] (a-1-4) edge (a-2-6);
\path[->] (a-2-1) edge node[below left]{$b$} (a-3-3);
\path[->] (a-3-3) edge (a-2-6);

\path[->] (a-2-1) edge node[left]{$p_1$} (a-7-1);
\path[->] (a-1-4) edge node[left]{$p_1$} (a-6-4);
\path[->] (a-3-3) edge node[left]{$p_1$} (a-8-3);
\path[->] (a-2-6) edge node[left]{$p_1$} (a-7-6);

\path[->] (a-7-1) edge node[above right]{$\alpha$} (a-6-4);
\path[->] (a-6-4) edge (a-7-6);
\path[->] (a-7-1) edge node[below left]{$\beta$} (a-8-3);
\path[->] (a-8-3) edge (a-7-6);

\end{tikzpicture}
\end{center}
and analogous diagrams for the projections $p_2$ and $p_3$.\\
\textbf{Claim 3}: the induced maps $\pi_0(a):\pi_0(\tY_3) \rarr \pi_0(p_{13}^*\tT)$ and $\pi_0(b):\pi_0(\tY_3) \rarr \pi_0(p_{12}^*\tT)$ are homeomorphisms. The homeomorphism $\psi_0 = \pi_0(b)\circ \pi_0(a)^{-1}$ lifts uniquely to an isomorphism $\psi: p_{13}^*\tT \rarr p_{12}^*\tT$ over $\tS_3$ and is equal to $p_1^*(\phi)$. Analogous statements hold respectively for the diagrams involving projections $p_2$ and $p_3$.\\
Proof of the claim: The proofs that $\pi_0(a), \pi_0(b)$ are homeomorphisms and that $\psi_0$ lifts canonically is virtually the same as the proof of the Claim 2. above. To see the last part of the claim we use the commutativity  of the last diagram and the fact that $\tT = \sqcup_t \tS$, from which we easily conclude that $\psi_0=\pi_0(p_1)^*(\phi_0)$ and from the definitions of $\phi$ and $\psi$ we see that also $p_1^*(\phi) = \psi$.

Having the claim, the cocycle condition for $\phi$ follows and thus we have constructed a descent datum on $\tT$. Moreover, by construction it is compatible with $\tY \rarr \tT$. Thus, by fpqc descent, we obtain an sheaf $T$ on $S_{fpqc}$ that becomes constant on $\tS$. Thus, we can view $T$ as an element of $\mathrm{Loc}_S$ (\cite[Def. 7.3.1]{BhattScholze}) and by the equivalence $\mathrm{Loc}_S=\rmCov_S$ of \cite[Lm. 7.3.9]{BhattScholze}, $T$ is representable by a geometric covering of $S$. The descent datum on $\tY \rarr \tT$ gives a morphism $Y \rarr T$ over $S$ by fpqc descent for morphisms of schemes (by \cite[Rmk. 040L]{StacksProject} and \cite[Lm. 02W0]{StacksProject}). See also Lm. \ref{desctoalgspace} and the preceding discussion. Let $\bt \in T$ be a geometric point. It lifts to a geometric point on $\tT$ (as $\tT \rarr T$ is a base-change of $\tS \rarr S$, and so weakly \'etale) and $Y_{\bt} = \tY \times_{\tT}\bt$. Thus, $Y_{\bt}$ is connected, as $\tY \rarr \tT$ had geometrically  connected fibres. Similarly, $Y \rarr T$ is surjective because $\tY \rarr \tT$ was. Thus, $T$ is is connected as an image of $Y$.

\end{proof}

\begin{rmk}\label{remark-indexing-set-finite}
Having finished the above proof, one can use Cor. \ref{application-of-hes-when-S-normal} together with Lm. \ref{finitegeomconncomp} to conclude that, when $S$ is normal,  the indexing set that appeared many times in the proof of Thm. \ref{steinforgeomcov}, namely $\pi_0(Y_{\bar{\eta}})$, is finite.
\end{rmk}

\begin{rmk}
  As $S$ is assumed to be Nagata, one is tempted to use \cite[Remark 7.3.10]{BhattScholze} to simplify the topological part of the proof (of having to deal with non-discrete $\pi_0$'s) at the expense of working with henselizations along more general closed subschemes. However, as the case when $S$ is normal is already non-trivial, this does not seem to be a more efficient approach.
\end{rmk}

\subsection{A remark on pro-\'etale descent}
The following results arose as an attempt of the author to understand why at some point of the proof of \cite[Lemma 7.3.9]{BhattScholze} the fpqc sheaf obtained via descent is automatically an algebraic space. The main question being of the existence of a cover by \'etale schemes. As explained to me via e-mail by Prof. Scholze, this is a consequence of the fact that the objects of $\rmLoc_X$ are classical (see Remark \ref{classical-remark}).

Below we present another approach that, however, requires an additional assumption that the maps in the pro-\'etale presentation of the considered cover are surjective. By following the construction (see e.g. \cite[Section 097Q]{StacksProject}) one sees that this can be guaranteed when constructing w-strictly local covers. This is for example sufficient for our needs in the proof of Theorem \ref{steinforgeomcov}. However, for w-contractible covers this condition will usually not be satisfied.
In the Prop. \ref{descentandlimits} and Lm. \ref{desctoalgspace} below, the results are Zariski local on the base $S$ and thus in the proofs we will assume it to be affine.
\begin{proposition}\label{descentandlimits}
Let $S$ be a qcqs scheme and let $\tS \rarr S$ be a pro-\'etale cover with a presentation as a limit over a directed inverse system $\tS = \invlim_\lambda S_\lambda$, where $S_\mu \rarr S_\lambda$ are all affine \'etale. \uline{Assume that all the maps $S_\mu \rarr S_\lambda$ in the inverse system are surjective.} Let $\tT \rarr \tS$ be a scheme with a fpqc-descent datum $\phi:p^*\tT \tilde{\longrightarrow} q^*\tT$ with respect to $\tS \rarr S$. For any $\lambda$, this gives a descent datum $\phi_\lambda$ for $\tT$ with respect to $\tS \rarr S_\lambda$ as well. Let $T$ (and $T_\lambda$) denote the fpqc-sheaf on $S$ ($S_\lambda$ respectively) obtained by the fpqc descent. Then for any affine open $\tU \subset \tT$ there exists $\lambda_0$ such that $\phi_{\lambda_0}$ induces a descent datum on $\tU \rarr \tT$ with respect to $\tS \rarr S_{\lambda_0}$, i.e. $\phi_{\lambda_0}$ restricts to an isomorphism $p_{\lambda_0}^*\tU \tilde{\longrightarrow} q_{\lambda_0}^*\tU$. Moreover, in the obtained morphism of fpqc sheaves $U_{\lambda_0} \rarr T_{\lambda_0}$, $U_{\lambda_0}$ is representable by a scheme.
\end{proposition}
\begin{proof}
Observe that there is a canonical isomorphism (coming from the diagonal morphism at each level) $\tS \simeq \invlim_\lambda \tS\times_{S_\lambda}\tS$. The isomorphisms $\phi_\lambda$ and the projections $p_\lambda, q_\lambda$ form inverse systems respectively and their limits are: $\invlim p_\lambda = \id_{\tS}$, $\invlim q_\lambda = \id_{\tS}$, $\invlim \phi_\lambda = \id_{\tT}$. Thus, in the limit, the $\phi_\lambda$'s do restrict to an isomorphism from $p_\lambda^*\tU$ to $q_\lambda^*\tU$. We need to show that this holds already for some $\lambda_0$. The schemes $V_\lambda = q_\lambda^* \tU$ and  $V_\lambda'=\phi_\lambda(p_\lambda^*\tU)$ are affine open subschemes of $q_\lambda^*\tT$ that become equal in the limit. The symmetric differences $\Delta_\lambda = (V_\lambda \setminus V_\lambda')\cup (V_\lambda' \setminus V_\lambda)$ are affine schemes and they form an inverse system. Indeed, if $\mu > \lambda$, then $V_\mu$ is the preimage of $V_\lambda$ via $q_\mu^*\tT \rarr q_\lambda^*\tT$ and analogously for $V_\mu'$. So $V_\mu' \setminus V_\mu$ maps to $V_\lambda' \setminus V_\lambda$ and similarly for $V_\mu \setminus V_\mu'$. Thus, $\Delta_\lambda$ form a directed inverse system of affine schemes with an empty limit. By \cite[Tag 01Z2]{StacksProject}, there exists  $\lambda_0$ such that $\Delta_{\lambda_0}$ is empty, which is precisely what we wanted to show. The last part follows from \cite[Tag 0247]{StacksProject}, because $\tU \subset \tT$ is an affine open, $\tT$ is quasi-separated and so $\tU \rarr \tT$ is quasi-affine.
\end{proof}
In what follows, let us work with the definition of an algebraic space as in the Stacks Project, i.e. \cite[Definition 025Y]{StacksProject}.

\begin{lemma}\label{desctoalgspace}
Let $S$ be a topologically noetherian scheme and let $T \in \mathrm{Loc}_S$ (notation as in \cite{BhattScholze}). Let $\tS \rarr S$ be a pro-\'etale cover satisfying the assumptions of Prop. \ref{descentandlimits}. Assume that the restriction $\tT=T_{|\tS}$ is a constant sheaf (i.e. can be written as $\sqcup_\alpha \tS$). Then $T$ is represented by an algebraic space.
\end{lemma}
\begin{proof}
$T$ is an fpqc sheaf on $S$ and in turn an fppf sheaf on $S$. We need to check that: a) the diagonal morphism $T \rarr T \times T$ is representable, b) there is an \'etale surjective morphism $U \rarr T$ with $U$ a scheme.
Let us start with showing b): let us cover $\tT$ with open affine subschemes $\tU_i$. As $\tT$ comes from $S$, it is equipped with a descent datum  with respect to $\tS \rarr S$. By Prop. \ref{descentandlimits}, for each $U_i$ there exist $\lambda_i$ such that $\tU_i \rarr \tT$ descends to a morphism $U_i \rarr T_{\lambda_i}$ with $U_i$ a scheme. Let $f_i$ be the composition of $U_i \rarr T_\lambda$ and $T_\lambda \rarr T$. We claim that this is an \'etale morphism of fpqc sheaves. Indeed, each morphism of the composition is representable, because each of them becomes a quasi-affine morphism of schemes after an fpqc base-change and thus we can apply \cite[Tag 0247]{StacksProject}. Similarly, both $U_i \rarr T_\lambda$ and $T_\lambda \rarr T$ are \'etale, because each of them becomes \'etale morphism of schemes after a suitable fpqc base-change. To show a) we again observe that $\Delta: T \rarr T \times T$ becomes $\tT \rarr \tT \times_{\tS} \tT$ after an fpqc base-change and thus a quasi-affine morphism of schemes (using that being quasi-affine can be checked on a chosen affine cover of the target we translate the problem to checking that the intersection of two affine opens in $\tT$ is quasi-compact, but this follows from quasi-separatedness of $\tT$ over $\tS$) and we apply \cite[Tag 0247]{StacksProject} to see that $\Delta$ is representable.
\end{proof}

\end{document}